\documentclass[11pt]{article}
\usepackage[tbtags]{amsmath}
\usepackage{amssymb}
\usepackage{amsthm}
\usepackage[misc]{ifsym}
\usepackage{cases}
\usepackage{mathrsfs}
\usepackage{wrapfig}
\usepackage{graphicx}
\usepackage{subfigure}
\usepackage{algorithmic}
\usepackage[ruled,vlined]{algorithm2e}
\usepackage{float}
\graphicspath{{./image/}}

\numberwithin{equation}{section}
\normalsize

\title{\bf Three-Level Multi-Leader-Follower Incentive Stackelberg Differential Game with $H_\infty$ Constraint \thanks{This work is supported by National Key R\&D Program of China (2022YFA1006104), National Natural Science Foundations of China (12471419, 12271304), and Shandong Provincial Natural Science Foundations (ZR2024ZD35, ZR2022JQ01).}}
\author{\normalsize  Na Xiang\thanks{\it School of Mathematics, Shandong University, Jinan 250100, P.R. China, E-mail: 202211967@mail.sdu.edu.cn} , Jingtao Shi\thanks{\it Corresponding author. School of Mathematics, Shandong University, Jinan 250100, P.R. China, E-mail: shijingtao@sdu.edu.cn}}

\date{}

\newtheorem{mypro}{Proposition}[section]
\newtheorem{mythm}{Theorem}[section]
\newtheorem{mydef}{Definition}[section]
\newtheorem{mylem}{Lemma}[section]
\newtheorem{Remark}{Remark}[section]
\newtheorem{Example}{Example}[section]

\begin{document}

\maketitle
	
\noindent{\bf Abstract:}\quad This paper is concerned with a three-level multi-leader-follower incentive Stackelberg game with $H_\infty$ constraint. Based on $H_2/H_\infty$ control theory, we firstly obtain the worst-case disturbance and the team-optimal strategy by dealing with a nonzero-sum stochastic differential game. The main objective is to establish an incentive Stackelberg strategy set of the three-level hierarchy in which the whole system achieves the top leader's team-optimal solution and attenuates the external disturbance under $H_\infty$ constraint. On the other hand, followers on the bottom two levels in turn attain their state feedback Nash equilibrium, ensuring incentive Stackelberg strategies while considering the worst-case disturbance. By convex analysis theory,  maximum principle and decoupling technique, the three-level incentive Stackelberg strategy set is obtained. Finally, a numerical example is given to illustrate the existence of the proposed strategy set.

\vspace{2mm}
	
\noindent{\bf Keywords:}\quad Incentive Stackelberg differential game, $H_\infty$ constraint, $H_2/H_\infty$ control, Nash equilibrium, team-optimal solution
	
\vspace{2mm}
	
\noindent{\bf Mathematics Subject Classification:}\quad 91A15, 91A65, 93E20
	
\section{Introduction}

The Stackelberg game, pioneered by von Stackelberg \cite{von Stackelberg52} in 1952, has important applications in various fields, such as economics, engineering management and computer science, etc.
Bagchi and Ba\c{s}ar \cite{Bagchi-Basar81} initially considered {\it stochastic linear-quadratic} (SLQ) nonzero-sum Stackelberg differential game and established existence and uniqueness of the Stackelberg solution, where the diffusion term of the state equation did not contain the state and control variables.
Yong \cite{Yong02} discussed the open-loop solution of SLQ nonzero-sum Stackelberg differential game, where the coefficients of the system  with both state-dependent and control-dependent noise are random, and the weight matrices for the controls in the cost functionals are not necessarily positive definite, and obtained the feedback representation of the open-loop equilibrium via a new stochastic Riccati equation. In recent two decades, stochastic Stackelberg differential game has been extensively investigated and there has been a great deal of literature around it. Bensoussan et al. \cite{Bensoussan-Chen-Sethi15} derived the maximum principle for the stochastic Stackelberg differential game with the control-independent diffusion term under different information structures.
Shi et al. \cite{Shi-Wang-Xiong16, Shi-Wang-Xiong17} studied SLQ Stackelberg differential games with asymmetric information.
Li et al. \cite{Li et al21} discussed the SLQ Stackelberg differential game under asymmetric information by a layered calculation method.
Li and Yu \cite{Li-Yu18} considered an SLQ generalized Stackelberg game with the multilevel hierarchy. Moon and Ba\c{s}ar \cite{Moon-Basar18} studied SLQ Stackelberg {\it mean field game} (MFG) with the adapted open-loop information structure of the leader, where one leader and arbitrarily large number of followers was considered, and solved by fixed-point approach, while Wang \cite{Wang24} tackled it by a direct approach.
Zheng and Shi \cite{Zheng-Shi22} discussed a Stackelberg stochastic differential game with asymmetric noisy observations.
Kang and Shi \cite{Kang-Shi22} studied a three-level SLQ Stackelberg differential game with asymmetric information.
Feng et al. \cite{Feng-Hu-Huang24} established a unified two-person differential decision setup, and studied the relationships between zero-sum SLQ Nash and Stackelberg differential game, local versus global information.
Wang and Wang \cite{Wang-Wang24} considered SLQ mean-field type partially observed Stackelberg differential game with two followers.
Li and Shi \cite{Li-Shi24} studied closed-loop solvability of SLQ mean-field type Stackelberg differential games.
	
In the Stackelberg game, the leader's strategy can induce the decision or action of the followers such that the leader's team-optimal solution can be achieved, which is called the {\it incentive} Stackelberg game, and this has been extensively studied for over 40 years.
Ho et al. \cite{Ho-Luh-Olsder82} investigated the deterministic and stochastic versions of the  incentive problem, and their relationship to economic literature is discussed.
Zheng and Ba\c{s}ar \cite{Zheng-Basar82} studied the existence and derivation of optimal affine incentive schemes for Stackelberg games with partial information by a geometric approach.
Zheng et al. \cite{Zheng-Basar-Cruz84} discussed applicability and appropriateness of a function-space approach in the derivation of causal real-time implementable optimal incentive Stackelberg strategies under various information patterns.
Mizukami and Wu \cite{Mizukami-Wu87,Mizukami-Wu88} considered the derivation of the sufficient conditions for the LQ incentive Stackelberg game with multi-players in a two-level hierarchy.
Ishida and Shimemura \cite{Ishida-Shimemura83} investigated the three-level incentive Stackelberg strategy in a nonlinear differential game, and derived the sufficient condition for a linear quadratic differential game, while Ishida \cite{Ishida87} considered different incentive strategies.
Li et al. \cite{Li-Cruz-Simaan02} studied the team-optimal state feedback Stackelberg strategy of a class of discrete-time two-person nonzero-sum LQ dynamic games.
Mukaidani et al. \cite{Mukaidani-Xu-Shima-Dragan17} discussed the incentive Stackelberg game for a class of Markov jump linear stochastic systems with multiple leaders and followers.
Lin et al. \cite{Lin-Gao-Zhang22} investigated the incentive feedback strategy for a class of stochastic Stackelberg games in finite horizon and infinite horizon.
Gao et al. \cite{Gao-Lin-Zhang23} considered the incentive feedback Stackelberg strategy for the discrete-time stochastic systems, then the incentive Stackelberg strategy for the discrete-time stochastic systems with mean-field terms is discussed in \cite{Gao-Lin-Zhang24}.
	
The incentive Stackelberg game is often combined with robust control theory. The incentive Stackelberg games under the $H_\infty$ constraint is based on $H_2/H_\infty$ control theory to solve.
Ahmed and Mukaidani \cite{Ahmed-Mukaidani16} studied  the incentive Stackelberg game for a class of  deterministic discrete-time system with a deterministic external disturbance.
For stochastic version, Mukaidani et al. \cite{Mukaidani-Xu-Dragan18} considered the incentive Stackelberg game with one leader and one follower subject to external disturbance by means of static output-feedback.
\cite{Mukaidani-Shima-Unno-Xu-Dragan17,Ahmed-Mukaidani-Shima17,Mukaidani-Xu18,Kawakami-Mukaidani-Xu-Tanaka18,Mukaidani-Xu19} discussed the incentive Stackelberg game with one leader and multiple non-cooperative followers subjected to the $H_\infty$ constraint.
Ahmed et al. \cite{Ahmed-Mukaidani-Shima17'} investigated multi-leader-follower incentive Stackelberg games for SLQ systems with $H_\infty$ constraint.
Mukaidani et al. \cite{Mukaidani-Xu-Shima-Ahmed18} considered the incentive Stackelberg game for a class of Markov jump SLQ systems with multi-leader-follower under $H_\infty$ constraint, where followers attain their state feedback Nash equilibrium and Pareto optimality.
Mukaidani et al. \cite{Mukaidani-Saravanakumar-Xu20} studied a robust static output feedback incentive Stackelberg game for a Markov jump SLQ system with multi-leader-follower, where the Pareto optimal strategies of the followers as cooperative strategy is chosen.
Mukaidani et al. \cite{Mukaidani-Irie-Xu-Zhang22} considered the static output feedback strategy for robust incentive Stackelberg games with a large population for mean-field stochastic systems.
	
In this paper, we study a three-level multi-leader-follower incentive Stackelberg game with $H_\infty$ constraint, which has practical significance, especially in corporate governance. The first level is specified as the Decision-making Level 1 with one person; the second level is lower than the first level and is specified as the Managerial Level 2 with two people; the third level is the lowest level of the whole system and is specified as the Executive Level 3 with three people (see Figure 1.(a)).
\begin{figure}[htbp]
\centering
\subfigure[]{
	\begin{minipage}[t]{0.5\linewidth}
		\centering
		\includegraphics[width=3in]{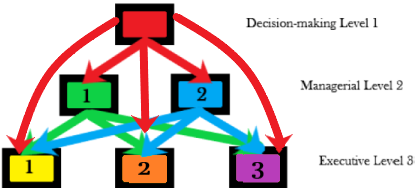}
	\end{minipage}%
}%
\subfigure[]{
	\begin{minipage}[t]{0.5\linewidth}
		\centering
		\includegraphics[width=2.15in]{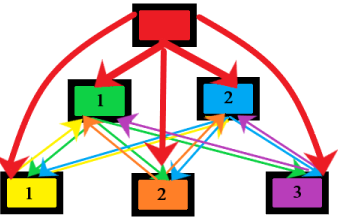}
	\end{minipage}%
}%
\centering
\caption{The three-level incentive Stackelberg game. }
\end{figure}
The arrows in the figure 1.(a) indicate the leader's incentives for followers. Since Decision-making Level 1 is in the highest position, who can incentivize each player in Managerial Level 2 and Executive Level 3 accordingly. Two people in Managerial Level 2 can only make incentives to the three people in the next level (Executive Level 3). For three people of Executive Level 3, they can only respond to the incentives of all the leaders (Decision-making Level 1 and Managerial Level 2). Figure 1.(b) depicts the incentives and responses of the three-level Stackelberg game system in detail.

Besides, due to the limited abilities and complicated information environment, the whole company doesn't have access to full information about the market. So it is quite natural to consider information uncertainty. Based on this, the drift term of the state equation contains the external stochastic disturbance $v(\cdot)$, which is unknown to all players in the three-level incentive Stackelberg game system. We address incentive and model uncertainty primarily through  $H_2/H_\infty$ control, viewing the disturbance as a new player and finding a Nash equilibrium solution based on the definition of team-optimal strategy. Then, according to the worst-case disturbance, the incentives between each level are carried out, and finally the team-optimal strategy of the game is reached. Our main contributions are summarized as follows.
\begin{itemize}
\item[$\bullet$] We investigate a class of three-level multi-leader-follower Stackelberg games with $H_\infty$ constraint, where the control variables and the external disturbance enter the diffusion term and drift term of the state equation, respectively.
\item[$\bullet$] We derive sufficient conditions for the three-level incentive Stackelberg game using information on follower's strategies, and prove that
three-level incentive matrices depend on an initial state value $x_0$.
\item[$\bullet$] After deriving the state feedback representation of the worst-case disturbance, we design its closed-loop form. Based on this, we find that the corresponding trajectory (i.e., $\tilde{x}(\cdot)$, $\hat{x}(\cdot)$) is equivalent to the optimal-team trajectory $\bar{x}(\cdot)$ after achieving incentive, which shows that the disturbance designed in advance is the worst-case disturbance $v^*(\cdot)$ at this time. Using the closed-loop form of the disturbance instead of the worst-case disturbance can avoid some complications, i.e., dealing with an SLQ optimal control problem in an augmented space or introducing inhomogeneous terms and linear terms of the state process and control processes into the state equation and cost functional, respectively.
\item[$\bullet$] Via convex optimization approach, we can achieve the open-loop Nash equilibrium for Managerial Level 2 and Executive Level 3 with corresponding incentives from leaders.
\item[$\bullet$] By $H_2/H_\infty$ control theory, the team-optimal strategy and the worst-case disturbance are derived at the same time, based on which the three-level incentive Stackelberg strategy is obtained.
\end{itemize}

The rest of this paper is organized as follows. Section 2 introduces some preliminary notations and formulations of the team-optimal solution and the finite horizon stochastic $H_2/H_\infty$ control problem. Section 3 discusses incentive Stackelberg equilibrium strategies. In Section 4, an algorithm procedure and a numerical example are used to elaborate the effectiveness of the proposed results. Section 5 concludes the paper.	
	
\section{Preliminaries}
	
Let ($\Omega,\mathcal{F},\mathbb{F},\mathbb{P}$) be a complete filtered probability space on which a standard one-dimensional Brownian motion $W=\left\lbrace W(t),0\leq t<\infty\right\rbrace$ is defined, where $\mathbb{F}=\left\lbrace\mathcal{F}_t\right\rbrace_{t\geq0}$ is the natural filtration of $W\left(\cdot\right)$ augmented by all the $\mathbb{P}$-null sets in $\mathcal{F}$. Let $\mathbb{R}^n$ denote the \emph{n}-dimensional Euclidean space with standard Euclidean norm $\left|\cdot\right| $ and standard Euclidean inner product $\left\langle\cdot,\cdot\right\rangle $. The transpose of a vector (or matrix) $x$ is denoted by $\mathbf{\emph{x}}^\top$. $\mbox{Tr}(A)$ denotes the trace of a square matrix $A$. Let $\mathbb{R}^{n\times m}$ be the Hilbert space consisting of all $n\times m$-matrices with the inner product $\left\langle A,B\right\rangle := \mbox{Tr}(AB^\top$) and the norm $\Vert A \Vert:=\langle A,A\rangle^\frac{1}{2}$. Denote the set of symmetric $n\times n$ matrices with real elements by $\mathbb{S}^n$. If $M\in\mathbb{S}^n$ is positive (semi-) definite, we write $M > (\geq) 0$. If there exists a constant $\delta>0$ such that $M\geq\delta I$, we write $M\gg0$.	
	
Consider a finite time horizon $[0,T]$ for a fixed $T>0$. Let $\mathbb{H}$ be a given Hilbert space, we denote
\begin{equation*}
\begin{aligned}
	L_{\mathbb{F}}^2(0,T;\mathbb{H})&:=\biggl\{\phi :[0,T]\times\Omega \mapsto \mathbb{H}\,\Big|\,\phi\mbox{ is }\mathbb{F}\mbox{-progressively measurable, such that}\\
    &\qquad \left( \mathbb{E}\int_0^T |\phi(s)|^2 ds\right) ^\frac{1}{2} < +\infty \biggr\}.
\end{aligned}
\end{equation*}
		
We consider the following controlled linear {\it stochastic differential equation} (SDE):
\begin{equation}\label{state}
\left\{
\begin{aligned}
	\mathrm{d}x(t)=&\Bigg[ A(t)x(t)+\sum_{i=1}^2 B_{1i}(t)u_{1i}(t)+\sum_{i=1}^2\sum_{j=1}^3 B_{2ij}(t)u_{2ij}(t)\\
    &\quad +\sum_{j=1}^3\sum_{i=1}^{2}B_{3ji}(t)u_{3ji}(t)+E(t)v(t)\Bigg] \mathrm{d}t\\
	&+\Bigg[ C(t)x(t)+\sum_{i=1}^2 D_{1i}(t)u_{1i}(t)+\sum_{i=1}^2\sum_{j=1}^3 D_{2ij}(t)u_{2ij}(t)\\
    &\qquad +\sum_{j=1}^3\sum_{i=1}^{2}D_{3ji}(t)u_{3ji}(t)\Bigg] \mathrm{d}W(t),\qquad t\in[0,T],\\
	x(0)=&\ x_0,\\
\end{aligned}
\right.
\end{equation}
where $A(\cdot),B_{1i}(\cdot),B_{2ij}(\cdot),B_{3ji}(\cdot),C(\cdot),D_{1i}(\cdot),D_{2ij}(\cdot),D_{3ji}(\cdot),E(\cdot)$ are deterministic and uniformly bounded functions on $[0,T]$ of proper dimensions. $v(\cdot)\in L_{\mathbb{F}}^2(0,T;\mathbb{R}^{n_v})$ represents the external unknown disturbance, $x(\cdot)\in\mathbb{R}^n$ is the state process and $x_0\in\mathbb{R}^n$ is the initial state. $u_{1i}(\cdot)\in\mathbb{R}^{m_{1i}}$ represents Decision-making Level 1's control input for the ith of Managerial Level 2; $u_{2ij}(\cdot)\in\mathbb{R}^{m_{2ij}}$ represents the ith of Managerial Level 2's control input for the jth of Executive Level 3; $u_{3ji}(\cdot)\in\mathbb{R}^{m_{3ji}}$ represents the jth of Executive Level 3's control input according to the ith of Managerial Level 2 in the sense of incentive Stackelberg strategy. Moreover, the index $i=1,2$ and $j=1,2,3$ denote the ith player of Managerial Level 2 and the jth player of Executive Level 3.

For the sake of simplicity, for $i=1,2$, $j=1,2,3$, let
\begin{equation*}
\begin{aligned}
	u_1(t)&:=\mathbf{col}\big[u_{11}(t)\;\; u_{12}(t)\big], \\
	u_{2i}(t)&:=\mathbf{col}\big[u_{2i1}(t)\;\; u_{2i2}(t)\;\; u_{2i3}(t)\big], \\
	u_{3j}(t)&:=\mathbf{col}\big[u_{3j1}(t)\;\; u_{3j2}(t)\big], \\
	u_{ci}(t)&:=\mathbf{col}\big[u_{2i1}(t)\;\; u_{2i2}(t)\;\; u_{2i3}(t)\;\; u_{31i}(t)\;\; u_{32i}(t)\;\; u_{33i}(t)\big].
\end{aligned}
\end{equation*}
The admissible control set $\mathcal{U}_1$ of Decision-making Level 1 is defined as follows:
\begin{equation*}
\begin{aligned}
    &\mathcal{U}_1:=\biggl\{u_1 :[0,T]\times\Omega \mapsto \mathbb{R}^{m_1}\,\Big|\,u_1\mbox{ is }\mathbb{F}\mbox{-progressively measurable, } \\
    &\qquad\qquad \mbox{such that }\left( \mathbb{E}\int_0^T |u_1(s)|^2 ds\right) ^\frac{1}{2} < +\infty, \mbox{ with } m_1=\sum_{i=1}^2 m_{1i}  \biggr\} ,
\end{aligned}
\end{equation*}
the admissible control set $\mathcal{U}_{2i}$ of the ith of Managerial Level 2 is
\begin{equation*}
\begin{aligned}
	&\mathcal{U}_{2i}:=\biggl\{u_{2i} :[0,T]\times\Omega \mapsto \mathbb{R}^{m_{2i}}\,\Big|\,u_{2i}\mbox{ is }\mathbb{F}\mbox{-progressively measurable, } \\
	&\qquad\qquad\mbox{such that }\left( \mathbb{E}\int_0^T |u_{2i}(s)|^2 ds\right) ^\frac{1}{2} < +\infty \mbox{ with } m_{2i}=\sum_{j=1}^3 m_{2ij}  \biggr\} ,
\end{aligned}
\end{equation*}
the admissible control set $\mathcal{U}_{3j}$ of the jth of Executive Level 3 is
\begin{equation*}
\begin{aligned}
	&\mathcal{U}_{3j}:=\biggl\{u_{3j} :[0,T]\times\Omega \mapsto \mathbb{R}^{m_{3j}}\,\Big|\,u_{3j}\mbox{ is }\mathbb{F}\mbox{-progressively measurable, } \\
	&\qquad\qquad\mbox{such that }\left( \mathbb{E}\int_0^T |u_{3j}(s)|^2 ds\right) ^\frac{1}{2} < +\infty \mbox{ with } m_{3j}=\sum_{i=1}^2 m_{3ji}  \biggr\} .
\end{aligned}
\end{equation*}

Under some mild conditions on the coefficients, for any $(x_0,u_1,u_{21},u_{22},u_{31},u_{32},u_{33},v)\in\mathbb{R}^n\times\mathcal{U}_1\times\mathcal{U}_{21}\times\mathcal{U}_{22}\times\mathcal{U}_{31}\times\mathcal{U}_{32}\times\mathcal{U}_{33}\times L_{\mathbb{F}}^2(0,T;\mathbb{R}^{n_v})$, there exists a unique (strong) solution $x(\cdot)\equiv x(\cdot;x_0,u_1,u_{21},u_{22},u_{31},u_{32},u_{33},v)\in L_{\mathbb{F}}^2(0,T;\mathbb{R}^{n})$ to (\ref{state}). Thus, we can define the cost functionals as follows.

For Decision-making Level 1:
\begin{equation}\label{cost_1}
\begin{aligned}
	&\qquad J^1(u_1,u_{21},u_{22},u_{31},u_{32},u_{33},v;x_0)\\
	&=\mathbb{E}\bigg\{ \int_0^T\bigg[ \left\langle Q_1(t)x(t),x(t)\right\rangle  +\sum_{i=1}^{2}\left\langle R_i(t)u_{1i}(t),u_{1i}(t)\right\rangle
    +\sum_{i=1}^{2}\sum_{j=1}^{3}\left\langle R_{1ij}(t)u_{2ij}(t),u_{2ij}(t)\right\rangle \\
	&\qquad\qquad\quad+\sum_{j=1}^3\sum_{i=1}^{2}\left\langle \bar{R}_{1ji}(t)u_{3ji}(t),u_{3ji}(t)\right\rangle\bigg] dt+\left\langle G_1x(T),x(T)\right\rangle \bigg\} ,
\end{aligned}
\end{equation}
for the ith of Managerial Level 2:
\begin{equation}\label{cost_2i}
\begin{aligned}
	&\qquad J^2_i(u_1,u_{21},u_{22},u_{31},u_{32},u_{33},v;x_0)\\
	&=\mathbb{E}\bigg\{ \int_0^T\bigg[ \left\langle Q_{2i}(t)x(t),x(t)\right\rangle  +\left\langle R_{2i}(t)u_{1i}(t),u_{1i}(t)\right\rangle
    +\sum_{j=1}^3\left\langle R_{2ij}(t)u_{2ij}(t),u_{2ij}(t)\right\rangle \\
	&\qquad\qquad\quad+\sum_{j=1}^3\left\langle \bar{R}_{2ji}(t)u_{3ji}(t),u_{3ji}(t)\right\rangle\bigg] dt+\left\langle G_{2i}x(T),x(T)\right\rangle \bigg\} ,
\end{aligned}
\end{equation}
for the jth player of Executive Level 3:
\begin{equation}\label{cost_3j}
\begin{aligned}
	&\qquad J^3_j(u_1,u_{21},u_{22},u_{31},u_{32},u_{33},v;x_0)\\
	&=\mathbb{E}\bigg\{ \int_0^T\bigg[ \left\langle Q_{3j}(t)x(t),x(t)\right\rangle  +\sum_{i=1}^{2}\left\langle R_{3ij}(t)u_{2ij}(t),u_{2ij}(t)\right\rangle\\
	&\qquad\qquad\quad+\sum_{i=1}^2\left\langle \bar{R}_{3ji}(t)u_{3ji}(t),u_{3ji}(t)\right\rangle\bigg] dt +\left\langle G_{3j}x(T),x(T)\right\rangle \bigg\} ,
\end{aligned}
\end{equation}
where $Q_1(\cdot),R_i(\cdot),R_{1ij}(\cdot),\bar{R}_{1ji}(\cdot),Q_{2i}(\cdot),R_{2i}(\cdot),R_{2ij}(\cdot),\bar{R}_{2ji}(\cdot),Q_{3j}(\cdot),R_{3ij}(\cdot),\bar{R}_{3ji}(\cdot)$ are deterministic and uniformly bounded symmetric matrix-valued functions on $[0,T]$ of proper dimensions. $G_1$, $G_{2i}$, $G_{3j}$ are $n\times n$ symmetric matrices. In addition, the weighting coefficients of cost functionals (\ref{cost_1})-(\ref{cost_3j}) satisfy the following:

{\bf(A1)} $Q_1(\cdot)\geq0$, $Q_{2i}(\cdot)\geq0$, $Q_{3j}(\cdot)\geq0$, $G_1\geq0$, $G_{2i}\geq0$, $G_{3j}\geq0$.

{\bf(A2)} $R_{i}(\cdot)\gg0$, $R_{1ij}(\cdot)\gg0$, $\bar{R}_{1ji}(\cdot)\gg0$, $R_{2i}(\cdot)\gg0$, $R_{2ij}(\cdot)\gg0$, $\bar{R}_{2ji}(\cdot)\gg0$, $R_{3ij}(\cdot)\gg0$, $\bar{R}_{3ji}(\cdot)\gg0$.
	

Next, we will introduce some definitions. First, we introduce the team-optimal solution concept (see \cite{Basar-Olsder99}), which is an important concept in this paper.
\begin{mydef}\label{def2.1}
Let $J^1(u_1,u_{21},u_{22},u_{31},u_{32},u_{33};x_0)$ be a given cost functional of the leader (Decision-making Level 1), where $u_1(\cdot)$ denotes the leader's control, and $(u_{21}(\cdot),u_{22}(\cdot),u_{31}(\cdot),\\u_{32}(\cdot),u_{33}(\cdot))$ denotes the followers' controls (Managerial Level 2 and Executive Level 3). A control set $(u_1^*,u_{21}^*,u_{22}^*,u_{31}^*,u_{32}^*,u_{33}^*)$ is known as the \textit{team-optimal solution} of this game if
	 \[J^1(u_1^*,u_{21}^*,u_{22}^*,u_{31}^*,u_{32}^*,u_{33}^*;x_0)\leq J^1(u_1,u_{21},u_{22},u_{31},u_{32},u_{33};x_0) ,\]
for any $u_1\in\mathcal{U}_1$, $u_{2i}\in\mathcal{U}_{2i}$, $u_{3j}\in\mathcal{U}_{3j}$.
\end{mydef}
If $J$ is quadratic and strictly convex on the product space $\mathcal{U}_1\times\mathcal{U}_{21}\times\mathcal{U}_{22}\times\mathcal{U}_{31}\times\mathcal{U}_{32}\times\mathcal{U}_{33}$, then a unique team-optimal solution exists.

Second, we introduce the finite horizon stochastic $H_2/H_\infty$ problem.
Consider the following stochastic linear system:
\begin{equation}\label{state-asst}
\left\{
\begin{aligned}
	&dx(t)=\big[A(t)x(t)+B(t)u(t)+E(t)v(t)\big]dt+A_1(t)x(t)dW(t), \\
	&x(0)=x_0 ,\\
	&z(t)=\mathbf{col}\big[C(t)x(t)\;D(t)u(t)\big], \qquad D^\top(t)D(t)=I_m,
\end{aligned}
\right.
\end{equation}
where all coefficient matrices are continuous functions of time with suitable dimensions. $x_0\in\mathbb{R}^n$ denotes the initial state, $u(\cdot)\in L_{\mathbb{F}}^2(0,T;\mathbb{R}^m)$ denotes the control input, $v(\cdot)\in L_{\mathbb{F}}^2(0,T;\mathbb{R}^{n_v})$ denotes the external disturbance, and $W(t)\in\mathbb{R}$ is a 1-dimensional standard Brownian motion defined in the filtered probability space ($\Omega,\mathcal{F},\mathbb{F},\mathbb{P}$). $z(\cdot)\in\mathbb{R}^{n_z}$ represents the controlled output.
Then, the finite horizon stochastic $H_2/H_\infty$ problem of (\ref{state-asst}) can be stated as follows (see \cite{Chen-Zhang04}).
\begin{mydef}\label{H2/H infity def}
Given the disturbance attenuation $\gamma>0$, $0\leq T<\infty$, to find a state feedback control $u^*(t,x)\in L_{\mathbb{F}}^2(0,T;\mathbb{R}^m)$ and the worst case disturbance $v^*(t,x)\in L_{\mathbb{F}}^2(0,T;\mathbb{R}^{n_v})$, such that
\textnormal{\mbox{ \ \ \ }1)}
\begin{equation*}
\begin{aligned}
	\Vert\mathcal{L}\Vert_{[0,T]}&=\underset{v(\cdot)\in L_{\mathbb{F}}^2(0,T;\mathbb{R}^{n_v}), v\neq0, x_0=0}{\sup} \frac{\Vert z\Vert_{[0,T]}}{\Vert v\Vert_{[0,T]}}  \\
	&=\underset{v(\cdot)\in L_{\mathbb{F}}^2(0,T;\mathbb{R}^{n_v}), v\neq0, x_0=0}{\sup}
    \frac{\left\lbrace \mathbb{E}\int_0^T[x^\top C ^\top Cx+{u^*}^\top u^*]dt\right\rbrace ^\frac{1}{2}}{\left\lbrace \mathbb{E}\int_0^Tv^\top vdt\right\rbrace ^\frac{1}{2}}<\gamma.
\end{aligned}
\end{equation*}
\textnormal{\mbox{ \ \ \ }2)}
When the worst case disturbance $v^*(t,x)\in L_{\mathbb{F}}^2(0,T;\mathbb{R}^{n_v})$, if it exists, is applied to (\ref{state-asst}), $u^*(\cdot)$ minimizes the output energy
\begin{equation}\label{cost Ju}
	J_u(u,v^*)=\mathbb{E}\int_{0}^{T}[x^\top C ^\top Cx+u^\top u]dt.
\end{equation}
Here, the so-called worst case disturbance $v^*$ means that
\begin{equation}
\begin{aligned}
	v^*(t,x)&=\underset{v\in L_{\mathbb{F}}^2(0,T;\mathbb{R}^{n_v})}{\arg\min}J_v(u^*,v) \\
	&=\underset{v\in L_{\mathbb{F}}^2(0,T;\mathbb{R}^{n_v})}{\arg\min}\mathbb{E}\int_{0}^{T}[\gamma^2v^\top v-z^\top z]dt,\quad\forall x_0\in\mathbb{R}^n.
\end{aligned}
\end{equation}
If the previous $(u^*,v^*)$ exists, then we say that the finite horizon $H_2/H_\infty$ control has \textit{Nash equilibrium solution} $(u^*,v^*)$, i.e.,
\begin{equation*}
	J_u(u^*,v^*)\leq J_u(u,v^*),\quad J_v(u^*,v^*)\leq J_v(u^*,v),
\end{equation*}
$u^*$ is a solution to the stochastic $H_2/H_\infty$ control, and $v^*$ is the corresponding worst-case disturbance.
\end{mydef}

In \cite{Limebeer94}, the finite horizon stochastic $H_2/H_\infty$ control problem can be formulated as a stochastic LQ nonzero-sum game.
	
Three-level incentive Stackelberg game under the $H_\infty$ constraint with multiple leaders and multiple followers, in this paper, is formulated as follows.

\textnormal{1)}
Given the disturbance attenuation level $\gamma>0$, find the team-optimal strategy of Decision-making Level 1 with $H_\infty$ constraint:
\begin{equation}
\begin{aligned}
	J^1(u_1^*,u_{21}^*,u_{22}^*,u_{31}^*,u_{32}^*,u_{33}^*,v^*;x_0)\leq J^1(u_1,u_{21},u_{22},u_{31},u_{32},u_{33},v^*;x_0), \\
	0\leq J_v(u_1^*,u_{21}^*,u_{22}^*,u_{31}^*,u_{32}^*,u_{33}^*,v^*;x_0)\leq J_v(u_1^*,u_{21}^*,u_{22}^*,u_{31}^*,u_{32}^*,u_{33}^*,v;x_0),
\end{aligned}
\end{equation}
for any $(u_1,u_{21},u_{22},u_{31},u_{32},u_{33})\in\mathcal{U}_{1}\times\mathcal{U}_{21}\times\mathcal{U}_{22}\times\mathcal{U}_{31}\times\mathcal{U}_{32}\times\mathcal{U}_{33}$, $v\neq0\in L_{\mathbb{F}}^2(0,T;\mathbb{R}^{n_v})$, where
\begin{equation}\label{cost Jv}
\begin{aligned}
	&\qquad J_v(u_1,u_{21},u_{22},u_{31},u_{32},u_{33},v;x_0)\\
	&:=\mathbb{E}\bigg\{ -\left\langle G_1x(T),x(T)\right\rangle +\int_0^T\bigg[\gamma^2\Vert v\Vert^2- \left\langle Q_1(t)x(t),x(t)\right\rangle  -\sum_{i=1}^2\left\langle R_i(t)u_{1i}(t),u_{1i}(t)\right\rangle\\
    &\qquad\quad-\sum_{i=1}^2\sum_{j=1}^3\left\langle R_{1ij}(t)u_{2ij}(t),u_{2ij}(t)\right\rangle -\sum_{j=1}^3\sum_{i=1}^{2}\left\langle \bar{R}_{1ji}(t)u_{3ji}(t),u_{3ji}(t)\right\rangle\bigg] dt\bigg\}.
\end{aligned}
\end{equation}

\textnormal{2)}
Decision-making Level 1 announces the following feedback strategy in advance to Managerial Level 2:
\begin{equation}\label{incentive form u_1i}
	u_{1i}(t)=\Gamma_{1i}\big(x(t),u_{2i}(t),u_{31i}(t),u_{32i}(t),u_{33i}(t),t\big).
\end{equation}
The parameters in (\ref{incentive form u_1i}) can be determined in the next step.

\textnormal{3)}
Find the Nash equilibrium strategies of Managerial Level 2:
\begin{equation}
\begin{aligned}
	\bar{J}^2_1(u^+_{c1},u^+_{c2};x_0)\leq\bar{J}^2_1(u_{c1},u^+_{c2};x_0),\\
	\bar{J}^2_2(u^+_{c1},u^+_{c2};x_0)\leq\bar{J}^2_2(u^+_{c1},u_{c2};x_0),
\end{aligned}
\end{equation}
where for $i=1,2$,
\begin{equation}
	\bar{J}^2_i(u_{c1},u_{c2};x_0):=J^2_i\big(\Gamma_1(x,,u_{2i},u_{31i},u_{32i},u_{33i}),u_{21},u_{22},u_{31},u_{32},u_{33},v^*(x);x_0\big),
\end{equation}
and $u_{ci}=\mathbf{col}\big[u_{2i1}\;u_{2i2}\;u_{2i3}\;u_{31i}\;u_{32i}\;u_{33i}\big]$, $\Gamma_1=\mathbf{col}\big[\Gamma_{11},\Gamma_{12}\big]$, $v^*(x)$ is the closed-loop form we designed based on the state feedback form of the worst-case disturbance $v^*(\cdot)$ solved in the first step. Decide parameters in (\ref{incentive form u_1i}) such that $u^+_{ci}(t)=u^*_{ci}(t)$, then the corresponding incentive strategy is denoted as $\Gamma^*_{1i}$.

\textnormal{4)}
Managerial Level 2 announce the following feedback strategy in advance to Executive Level 3:
\begin{equation}\label{incentive form u_2ij}
	u_{2ij}(t)=\Gamma_{2ij}\big(x(t),u_{3ji}(t),t\big).
\end{equation}
The parameters in (\ref{incentive form u_2ij}) can be determined in the next step.

\textnormal{5)}
Find the Nash equilibrium strategies of Executive Level 3:
\begin{equation}
\begin{aligned}
	\bar{J}^3_j(\bar{u}_{31},\bar{u}_{32},\bar{u}_{33};x_0)\leq\bar{J}^3_j\big(\bar{u}_{-3j}(u_{3j});x_0\big),
\end{aligned}
\end{equation}
where for $j=1,2,3$,
\begin{equation}
	\bar{J}^3_j(u_{31},u_{32},u_{33};x_0):=J^3_j\big(\Gamma_{11}^*,\Gamma_{12}^*,(\Gamma_{2ij})_{i,j},u_{31},u_{32},u_{33},v^*(x);x_0\big),
\end{equation}
and $u_{3j}=\mathbf{col}\big[u_{3j1}\;u_{3j2}\big]$, $\bar{u}_{-3j}(u_{3j})$ denotes the control
profile which include $u_{3j}$ and $\bar{u}_{3k}$ ($k=1,2,3$, $k\neq j$). Decide parameters in (\ref{incentive form u_2ij}) such that $\bar{u}_{3j}(t)=u^*_{3j}(t)$, then the corresponding incentive strategy is denoted as $\Gamma^*_{2ij}$.

In this article, we may suppress time index t if it causes no confusion.

\section{Incentive Stackelberg equilibrium strategies}

\subsection{Team-optimal strategy}

As mentioned above, the state equation is (\ref{state}), and the cost functional of Decision-making Level 1 is (\ref{cost_1}). Due to the definition of team-optimal solution,  for simplicity, we introduce the following notation, centralizing the control inputs in the stochastic system (\ref{state}) and (\ref{cost_1}). The following centralized stochastic system can be introduced:
\begin{equation}\label{c-state}
\left\{
\begin{aligned}
	dx(t)&=\big[A(t)x(t)+B_c(t)u_c(t)+E(t)v(t)\big]dt+\big[C(t)x(t)+D_c(t)u_c(t)\big]dW(t), \\
	x(0)&=x_0,
\end{aligned}
\right.
\end{equation}
together with
\begin{equation}\label{J1c}
	J^1(u_c,v;x_0)=\mathbb{E}\bigg\{ \int_0^T\big[ \left\langle Q_1(t)x(t),x(t)\right\rangle+\left\langle R_c(t)u_c(t),u_c(t)\right\rangle\big]dt+\left\langle G_1x(T),x(T)\right\rangle \bigg\} ,
\end{equation}
\begin{equation}\label{Jvc}
	J_v(u_c,v;x_0)=\mathbb{E}\bigg\{ \int_0^T\big[\gamma^2\Vert v\Vert^2-\left\langle Q_1(t)x(t),x(t)\right\rangle-\left\langle R_c(t)u_c(t),u_c(t)\right\rangle\big]dt-\left\langle G_1x(T),x(T)\right\rangle \bigg\} ,
\end{equation}
where
\begin{equation*}
\begin{aligned}
	B_c&:=\big[B_{11}\;B_{12}\;B_{211}\;B_{212}\;B_{213}\;B_{221}\;B_{222}\;B_{223}\;B_{311}\;B_{312}\;B_{321}\;B_{322}\;B_{331}\;B_{332}\big],\\
	D_c&:=\big[D_{11}\;D_{12}\;D_{211}\;D_{212}\;D_{213}\;D_{221}\;D_{222}\;D_{223}\;D_{311}\;D_{312}\;D_{321}\;D_{322}\;D_{331}\;D_{332}\big],\\
	R_c&:=\mathbf{block}\;\mathbf{diag}\big(R_1\;R_2\;R_{111}\;R_{112}\;R_{113}\;R_{121}\;R_{122}\;R_{123}\;\bar{R}_{111}\;\bar{R}_{112}\;\bar{R}_{121}\;\bar{R}_{122}\;\bar{R}_{131}\;\bar{R}_{132}\big),\\
	u_c&:=\mathbf{col}\big[u_1\;u_{21}\;u_{22}\;u_{31}\;u_{32}\;u_{33}\big].
\end{aligned}
\end{equation*}

\begin{mypro}\label{prop-3.1}
Let (A1)-(A2) hold, and assume the attenuation index $\gamma$ is sufficiently large. Then, $(u^*_c(\cdot),v^*(\cdot))$ is an open-loop Nash equilibrium if and only if there exists a (unique) triple $(\bar{x}(\cdot),\bar{y}(\cdot),\bar{z}(\cdot))$ satisfying the FBSDE:
\begin{equation}\label{Hamiltonian-1}	
\left\{\begin{aligned}
	d\bar{x}&=\big[A\bar{x}-(B_cR^{-1}_cB_c^\top-\gamma^{-2}EE^\top)\bar{y}-B_cR^{-1}_cD^\top_c\bar{z}\big]dt\\
    &\quad +\big[C\bar{x}-D_cR^{-1}_cB^\top_c\bar{y}-D_cR^{-1}_cD^\top_c\bar{z}\big]dW, \\
	d\bar{y}&=-\big[A^\top\bar{y}+C^\top\bar{z}+Q_1\bar{x}\big]dt+\bar{z}(t)dW, \\
	\bar{x}(0)&=x_0, \quad \bar{y}(T)=G_1\bar{x}(T).
\end{aligned}\right.
\end{equation}
Moreover, the open-loop Nash equilibrium $(u^*_c(\cdot),v^*(\cdot))$ is given by
\begin{equation}\label{olne}
\begin{aligned}
	u^*_c(t)&=-R_c^{-1}(t)\big[B^\top_c(t)\bar{y}(t)+D^\top_c(t)\bar{z}(t)\big] , \\
	v^*(t)&=\gamma^{-2}E^\top(t)\bar{y}(t),
\end{aligned}
\end{equation}
and the corresponding optimal functionals are
\begin{equation*}
\begin{aligned}
	J^1(u^*_c,v^*;x_0)=&\left\langle x_0,\bar{y}(0)\right\rangle+\gamma^{-2}\mathbb{E} \int_0^T\langle  E(t)E^\top(t) \bar{y}(t),\bar{y}(t)\rangle  dt,\\
	J_v(u^*_c,v^*;x_0)=&-\left\langle x_0,\bar{y}(0)\right\rangle.
\end{aligned}
\end{equation*}
\end{mypro}

\begin{proof}
By Theorem 2.2.1 of Sun et al. \cite{Sun-Yong20}, the proof is omitted here.
\end{proof}

In the following, by applying the decoupling technique, we have the closed-loop representation of open-loop Nash equilibrium.
\begin{mypro}\label{prop-3.2}
 Suppose that the following equation
\begin{equation}\label{P}
\left\{\begin{aligned}
	&\dot{P}+PA+A^\top P-P(B_cR^{-1}_cB_c^\top-\gamma^{-2}EE^\top)P+Q_1 \\
	&\quad+(C-D_cR^{-1}_cB^\top_cP)^\top(I+PD_cR^{-1}_cD^\top_c)^{-1}P(C-D_cR^{-1}_cB^\top_cP)=0,\\
	&P(T)=G_1,
\end{aligned}\right.
\end{equation} 	
admits a solution $P(\cdot)$ such that $I+PD_cR^{-1}_cD^\top_c$ is invertible, then the Nash equilibrium (\ref{olne}) has the following
representation:
\begin{equation}\label{cor}
\begin{aligned}
	u^*_c(t)&=-R_c^{-1}\big[B^\top_cP+D^\top_c(I+PD_cR^{-1}_cD^\top_c)^{-1}P(C-D_cR^{-1}_cB^\top_cP)\big]\bar{x}(t) ,  \\
	v^*(t)&=\gamma^{-2}E^\top P\bar{x}(t).
\end{aligned}
\end{equation}
\end{mypro}

\begin{proof}
Let $\bar{y}=P\bar{x}$, applying It\^o's formula to it and comparing the coefficients with (\ref{Hamiltonian-1}), we have
\begin{equation}\label{bar z}
	\bar{z}(t)=(I+PD_cR^{-1}_cD^\top_c)^{-1}P(C-D_cR^{-1}_cB^\top_cP)\bar{x}(t),
\end{equation}
and $P(\cdot)$ satisfies (\ref{P}). Therefore, the Nash equilibrium (\ref{olne}) has the representation (\ref{cor}).
By $\bar{y}=P\bar{x}$ and (\ref{bar z}), we have the following equation satisfied by the optimal trajectory $\bar{x}(\cdot)$.
\begin{equation}\label{bar x}
	d\bar{x}(t)=\mathbf{a}(t)\bar{x}(t)dt+\mathbf{b}(t)\bar{x}(t)dW(t),\qquad \bar{x}(0)=x_0,
\end{equation}
where
\begin{equation*}
\begin{aligned}
	\mathbf{a}&:=A-(B_cR^{-1}_cB_c^\top-\gamma^{-2}EE^\top)P-B_cR^{-1}_cD^\top_c(I+PD_cR^{-1}_cD^\top_c)^{-1}P(C-D_cR^{-1}_cB^\top_cP),\\
	\mathbf{b}&:=C-D_cR^{-1}_cB^\top_cP-D_cR^{-1}_cD^\top_c(I+PD_cR^{-1}_cD^\top_c)^{-1}P(C-D_cR^{-1}_cB^\top_cP).
\end{aligned}
\end{equation*}
\end{proof}

\begin{Remark}
When the initial state $x_0=0$, by (\ref{bar x}) and (\ref{cor}), $\bar{x}(t)\equiv0$, $u^*_c(t)\equiv0$, and $v^*(t)\equiv0$, from which we can get
	\[J_v(u^*_c,v^*;0)=0.\]
Therefore, by the definition of the finite horizon $H_2/H_\infty$ control problem, $(u^*_c,v^*)$ given by (\ref{cor}) is a pair of solutions. $u^*$ is $H_2/H_\infty$ optimal control, and $v^*$ is the worst-case disturbance.
\end{Remark}

For notational simplicity, we set
\[\Lambda:=(I+PD_cR^{-1}_cD^\top_c)^{-1}P(C-D_cR^{-1}_cB^\top_cP),\]
then
\[u^*_c(t)=-R_c^{-1}[B^\top_cP+D^\top_c\Lambda]\bar{x}(t) .\]
By $R_c$, $B_c$, $D_c$ and $u_c$, we can display and represent all elements in $u^*_c(\cdot)$ as follows:
\begin{equation}\label{team-optimal}
\begin{aligned}	
	u^*_c&=\mathbf{col}\big[u^*_1\;u^*_{21}\;u^*_{22}\;u^*_{31}\;u^*_{32}\;u^*_{33}\big],\\
	&=\mathbf{col}\big[u^*_{11}\;u^*_{12}\;u^*_{211}\;u^*_{212}\;u^*_{213}\;u^*_{221}\;u^*_{222}\;u^*_{223}\;u^*_{311}\;u^*_{312}\;u^*_{321}\;u^*_{322}\;u^*_{331}\;u^*_{332}\;\big],\\
	u^*_{1i}(t)&=-R^{-1}_i(t)\big[B^\top_{1i}(t)P(t)+D^\top_{1i}(t)\Lambda(t)\big]\bar{x}(t),\\
	u^*_{2ij}(t)&=-R^{-1}_{1ij}(t)\big[B^\top_{2ij}(t)P(t)+D^\top_{2ij}(t)\Lambda(t)\big]\bar{x}(t),\\
	u^*_{3ji}(t)&=-\bar{R}^{-1}_{1ji}(t)\big[B^\top_{3ji}(t)P(t)+D^\top_{3ji}(t)\Lambda(t)\big]\bar{x}(t).
\end{aligned}
\end{equation}
In other words, (\ref{team-optimal}) is
the team-optimal solution with $H_\infty$ constraint in three-level incentive system, and the worst-case disturbance is $v^*(t)=\gamma^{-2}E^\top P\bar{x}(t)$.

For the sake of the discussion that follows, we set
\begin{equation*}
\begin{aligned}
	R_{1i}&:=\mathbf{block}\;\mathbf{diag}(R_{1i1}\;R_{1i2}\;R_{1i3}\;\bar{R}_{11i}\;\bar{R}_{12i}\;\bar{R}_{13i}),\\
	B_i&:=\big[B_{2i1}\;B_{2i2}\;B_{2i3}\;B_{31i}\;B_{32i}\;B_{33i}\;\big],\quad D_i:=[D_{2i1}\;D_{2i2}\;D_{2i3}\;D_{31i}\;D_{32i}\;D_{33i}\;],
\end{aligned}
\end{equation*}
then
\begin{equation}\label{u*ci}
	u^*_{ci}(t)=-R^{-1}_{1i}(t)\big[B^\top_{i}(t)P(t)+D^\top_{i}(t)\Lambda(t)\big]\bar{x}(t).
\end{equation}

\subsection{Nash equilibrium of Managerial Level 2}

Based on the incentive given by Decision-making Level 1 ahead of time, two people in Managerial Level 2 determine their strategies to achieve a Nash equilibrium by
responding to the announced strategy of Decision-making Level 1. Assume incentive form of Decision-making Level 1 is the following feedback strategy:

\begin{equation}\label{Gamma 1i}
\begin{aligned}
	u_{1i}(t)&=\Gamma_{1i}\big(x(t),u_{2i}(t),u_{31i}(t),u_{32i}(t),u_{33i}(t),t\big)\\
	&=-R^{-1}_i(B^\top_{1i}P+D^\top_{1i}\Lambda)x+\sum_{j=1}^3\eta_{1ij}\big[u_{2ij}+R^{-1}_{1ij}(B^\top_{2ij}P+D^\top_{2ij}\Lambda)x\big]\\
    &\quad +\sum_{j=1}^3\zeta_{1ij}\big[u_{3ji}+\bar{R}^{-1}_{1ji}(B^\top_{3ji}P+D^\top_{3ji}\Lambda)x\big]\\
	&=\bigg[-R^{-1}_i(B^\top_{1i}P+D^\top_{1i}\Lambda)+\sum_{j=1}^3\eta_{1ij}R^{-1}_{1ij}(B^\top_{2ij}P+D^\top_{2ij}\Lambda)\\
    &\qquad +\sum_{j=1}^3\zeta_{1ij}\bar{R}^{-1}_{1ji}(B^\top_{3ji}P+D^\top_{3ji}\Lambda)\bigg]x +\sum_{j=1}^{3}\eta_{1ij}u_{2ij}+\sum_{j=1}^3\zeta_{1ij}u_{3ji}\\
	&=\sum_{j=1}^3\xi_{1ij}(t)x(t)+\sum_{j=1}^{3}\eta_{1ij}(t)u_{2ij}(t)+\sum_{j=1}^3\zeta_{1ij}(t)u_{3ji}(t),\\
\end{aligned}
\end{equation}
where
\[\xi_{1ij}:=-R^{-1}_i(B^\top_{1i}P+D^\top_{1i}\Lambda)+\eta_{1ij}R^{-1}_{1ij}(B^\top_{2ij}P+D^\top_{2ij}\Lambda)+\zeta_{1ij}\bar{R}^{-1}_{1ji}(B^\top_{3ji}P+D^\top_{3ji}\Lambda),\]
and $\xi_{1ij}(t)\in\mathbb{R}^{m_{1i}\times n}$ $\eta_{1ij}(t)\in\mathbb{R}^{m_{1i}\times m_{2ij}}$, $\zeta_{1ij}(t)\in\mathbb{R}^{m_{1i}\times m_{3ji}}$ are (unknown) strategy parameter matrices whose components are continuous functions of $t$ on the interval $[0,T]$ respectively. Notice that
\begin{equation}\label{u*_1i}
\begin{aligned}
	u^*_{1i}(t)&=\Gamma_{1i}\big(\bar{x}(t),u^*_{2i}(t),u^*_{31i}(t),u^*_{32i}(t),u^*_{33i}(t),t\big) \\
	&=\sum_{j=1}^3\xi_{1ij}(t)\bar{x}(t)+\sum_{j=1}^3\eta_{1ij}(t)u^*_{2ij}(t)+\sum_{j=1}^3\zeta_{1ij}(t)u^*_{3ji}(t).
\end{aligned}
\end{equation}

Design the worst-case disturbance $v^*(\cdot)$ has the following closed-loop form
\begin{equation}\label{v}
	v^*(x)(t)=\gamma^{-2}(t)E^\top(t)P(t)x(t).
\end{equation}
Then we substitute (\ref{Gamma 1i}) and (\ref{v}) into state equation (\ref{state}) and cost functional for the ith of Managerial Level 2 (\ref{cost_2i}). By simplification, we have
\begin{equation}\label{state equation2}
\left\{\begin{aligned}
    dx&=\bigg[\bar{A}x+\sum_{i=1}^2\sum_{j=1}^3\bar{B}_{2ij}u_{2ij}+\sum_{j=1}^3\sum_{i=1}^2\bar{B}_{3ji}u_{3ji}\bigg]dt\\
    &\quad +\bigg[\bar{C}x+\sum_{i=1}^2\sum_{j=1}^{3}\bar{D}_{2ij}u_{2ij}+\sum_{j=1}^3\sum_{i=1}^{2}\bar{D}_{3ji}u_{3ji}\bigg]dW, \\
	&=\bigg[\bar{A}x+\sum_{i=1}^2\bar{B}_{ci}u_{ci}\bigg]dt+\bigg[\bar{C}x+\sum_{i=1}^2\bar{D}_{ci}u_{ci}\bigg]dW,\\
	x(0)&=x_0,
\end{aligned}\right.
\end{equation}
\begin{equation}\label{cost 2}
\begin{aligned}
	\bar{J}^2_i(u_{c1},u_{c2};x_0)&=\mathbb{E}\bigg\{ \int_0^T\big[ \left\langle \bar{Q}_{2i}(t)x(t),x(t)\right\rangle +\,2 \left\langle x(t),S_{2i}(t)u_{ci}(t)\right\rangle \\
	&\qquad\qquad +\left\langle R_{ci}(t)u_{ci}(t),u_{ci}(t)\right\rangle\big]dt+\left\langle G_{2i}x(T),x(T)\right\rangle \bigg\} ,
\end{aligned}
\end{equation}
where
\begin{equation*}
\begin{aligned}
	u_{ci}(t)&:=\mathbf{col}\big[u_{2i1}(t)\;\; u_{2i2}(t)\;\; u_{2i3}(t)\;\; u_{31i}(t)\;\; u_{32i}(t)\;\; u_{33i}(t)\big],\\
	\bar{A}&:=A+\gamma^{-2}EE^\top P+\sum_{i=1}^2\sum_{j=1}^3B_{1i}\xi_{1ij},\quad \bar{C}:=C+\sum_{i=1}^2\sum_{j=1}^3D_{1i}\xi_{1ij},\\
	\bar{B}_{2ij}&:=B_{2ij}+B_{1i}\eta_{1ij},\quad \bar{B}_{3ji}:=B_{3ji}+B_{1i}\zeta_{1ij},\quad \bar{B}_{ci}:=\big[\bar{B}_{2i1}\;\bar{B}_{2i2}\;\bar{B}_{2i3}\;\bar{B}_{31i}\;\bar{B}_{32i}\;\bar{B}_{33i}\big],\\
	\bar{D}_{2ij}&:=D_{2ij}+D_{1i}\eta_{1ij},\quad \bar{D}_{3ji}:=D_{3ji}+D_{1i}\zeta_{1ij},\quad \bar{D}_{ci}:=\big[\bar{D}_{2i1}\;\bar{D}_{2i2}\;\bar{D}_{2i3}\;\bar{D}_{31i}\;\bar{D}_{32i}\;\bar{D}_{33i}\big],\\
	\bar{Q}_{2i}&:=Q_{2i}+\bigg(\sum_{j=1}^3\xi_{1ij}^\top\bigg)R_{2i}\bigg(\sum_{j=1}^3\xi_{1ij}\bigg),\quad
    S_{2i}:=\sum_{j=1}^{3}\xi_{1ij}^\top R_{2i}\big[\eta_{1i1}\;\eta_{1i2}\;\eta_{1i3}\;\zeta_{1i1}\;\zeta_{1i2}\;\zeta_{1i3}\big],\\
\end{aligned}
\end{equation*}
\begin{equation}\label{notation-1}
\begin{aligned}
	R^1_{ci}&:=\begin{pmatrix}
		R_{2i1}+\eta^\top_{1i1}R_{2i}\eta_{1i1}& \eta^\top_{1i1}R_{2i}\eta_{1i2} & \eta^\top_{1i1}R_{2i}\eta_{1i3}\\
		\eta^\top_{1i2}R_{2i}\eta_{1i1}& R_{2i2}+\eta^\top_{1i2}R_{2i}\eta_{1i2} & \eta^\top_{1i2}R_{2i}\eta_{1i3}\\
		\eta^\top_{1i3}R_{2i}\eta_{1i1}& \eta^\top_{1i3}R_{2i}\eta_{1i2} & R_{2i3}+\eta^\top_{1i3}R_{2i}\eta_{1i3}
	\end{pmatrix},\\
	R^2_{ci}&:=
	\begin{pmatrix}
		\eta^\top_{1i1}R_{2i}\zeta_{1i1} & \eta^\top_{1i1}R_{2i}\zeta_{1i2} & \eta^\top_{1i1}R_{2i}\zeta_{1i3}\\
		\eta^\top_{1i2}R_{2i}\zeta_{1i1} & \eta^\top_{1i2}R_{2i}\zeta_{1i2} & \eta^\top_{1i2}R_{2i}\zeta_{1i3}\\
		\eta^\top_{1i3}R_{2i}\zeta_{1i1} & \eta^\top_{1i3}R_{2i}\zeta_{1i2} & \eta^\top_{1i3}R_{2i}\zeta_{1i3}\\
	\end{pmatrix},\qquad
	R_{ci}=
	\begin{pmatrix}
		R^1_{ci}&R^2_{ci}\\
		R^3_{ci}&R^4_{ci}\\
	\end{pmatrix},\\
	R^3_{ci}&:=
	\begin{pmatrix}
		\zeta^\top_{1i1}R_{2i}\eta_{1i1}& \zeta^\top_{1i1}R_{2i}\eta_{1i2} & \zeta^\top_{1i1}R_{2i}\eta_{1i3} \\
		\zeta^\top_{1i2}R_{2i}\eta_{1i1}& \zeta^\top_{1i2}R_{2i}\eta_{1i2} & \zeta^\top_{1i2}R_{2i}\eta_{1i3} \\
		\zeta^\top_{1i3}R_{2i}\eta_{1i1}& \zeta^\top_{1i3}R_{2i}\eta_{1i2} & \zeta^\top_{1i3}R_{2i}\eta_{1i3} \\
	\end{pmatrix} ,\\
	R^4_{ci}&:=
	\begin{pmatrix}
		\bar{R}_{21i}+\zeta^\top_{1i1}R_{2i}\zeta_{1i1} & \zeta^\top_{1i1}R_{2i}\zeta_{1i2} & \zeta^\top_{1i1}R_{2i}\zeta_{1i3}\\
		\zeta^\top_{1i2}R_{2i}\zeta_{1i1} & \bar{R}_{22i}+\zeta^\top_{1i2}R_{2i}\zeta_{1i2} & \zeta^\top_{1i2}R_{2i}\zeta_{1i3}\\
		\zeta^\top_{1i3}R_{2i}\zeta_{1i1} & \zeta^\top_{1i3}R_{2i}\zeta_{1i2} &\bar{R}_{23i}+\zeta^\top_{1i3}R_{2i}\zeta_{1i3} \\
	\end{pmatrix}.
\end{aligned}
\end{equation}

Next, Managerial Level 2 need to find the Nash equilibrium
$$(u^+_{c1}(\cdot),u^+_{c2}(\cdot))\in L_{\mathbb{F}}^2(0,T;\mathbb{R}^{\sum_{j=1}^{3}(m_{21j}+m_{3j1})})\times L_{\mathbb{F}}^2(0,T;\mathbb{R}^{\sum_{j=1}^{3}(m_{22j}+m_{3j2})}) ,$$
such that
\begin{equation}
\begin{aligned}
	\bar{J}^2_1(u^+_{c1},u^+_{c2};x_0)\leq\bar{J}^2_1(u_{c1},u^+_{c2};x_0),\\
	\bar{J}^2_2(u^+_{c1},u^+_{c2};x_0)\leq\bar{J}^2_2(u^+_{c1},u_{c2};x_0),
\end{aligned}
\end{equation}
for any $(u_{c1}(\cdot),u_{c2}(\cdot))\in L_{\mathbb{F}}^2(0,T;\mathbb{R}^{\sum_{j=1}^{3}(m_{21j}+m_{3j1})})\times L_{\mathbb{F}}^2(0,T;\mathbb{R}^{\sum_{j=1}^{3}(m_{22j}+m_{3j2})})$.

\begin{Remark}
Although the worst-case disturbance $v^*(\cdot)$ is not substituted into the state equation, but rather (\ref{v}), we will see later that $v^*(\tilde{x})(\cdot)$ is indeed the worst-case disturbance when the incentive form (\ref{Gamma 1i}) becomes the incentive strategy set (Remark \ref{remark-3.4}), i.e.,
	\[\qquad\qquad v^*(\tilde{x})(t)=v^*(t),\quad t\in[0,T],\quad\mathbb{P}\mbox{-}a.s.\]
\end{Remark}

\begin{mydef}
Let $F(g)$ be a \textit{real-valued functional} of $g\in L_{\mathbb{F}}^2(0,T;\mathbb{R}^n)$. If $F(g)\geq 0$ for all $g\in L_{\mathbb{F}}^2(0,T;\mathbb{R}^n)$, $F(\cdot)$ is said to be \textit{positive semidefinite}. If furthermore, $F(g)> 0$ for all $g\neq 0$, $F$ is said to be \textit{positive definite}.
\end{mydef}

\begin{mylem}\label{lemma3.1}
Let (A1)-(A2) hold. For any $x_0\in\mathbb{R}^n$, $\bar{J}^2_i(u^+_{-ci}(u_{ci});x_0)$ is convex (resp., strictly convex) in $u_{ci}(\cdot)\in L_{\mathbb{F}}^2(0,T;\mathbb{R}^{\sum_{j=1}^{3}(m_{2ij}+m_{3ji})})$ if and only if $J_i'(g_i(\cdot))$ is positive semidefinite (resp., positive definite), where
$$
	J_i'(g_i(\cdot)):=\mathbb{E}\left\lbrace \int_0^T\left[\left\langle \bar{Q}_{2i}z'_i,z'_i\right\rangle +\,2 \left\langle z'_i,S_{2i}g_i\right\rangle+\left\langle R_{ci}g_i,g_i\right\rangle\right] dt
    +\left\langle G_{2i}z'_i(T),z'_i(T)\right\rangle \right\rbrace,
$$
and $z'_i(\cdot)$ satisfies:
\begin{equation}\label{auxiliary J'}
\left\{\begin{aligned}
	dz'_i(t)&=[\bar{A}z'_i+\bar{B}_{ci}g_i]dt+[\bar{C}z'_i+\bar{D}_{ci}g_i]dW ,\\
	z'_i(0)&=0.
\end{aligned}\right.
\end{equation}
\end{mylem}

\begin{proof}
We consider the case $i=1$, and the proof is the same when $i=2$. For any $x_0\in\mathbb{R}^n$, fix $u^+_{c2}$, let $x_1(\cdot)$, $x_2(\cdot)$ be the states of (\ref{state equation2}) corresponding $u^1_{c1}(\cdot)$, $u^2_{c1}(\cdot)$, respectively. Taking any $\lambda_1\in[0,1]$ and denoting $\lambda_2:=1-\lambda_1$, we get
\begin{equation*}
\begin{aligned}
	&\lambda_1\bar{J}^2_1(u^1_{c1},u^+_{c2};x_0)+\lambda_2\bar{J}^2_1(u^2_{c1},u^+_{c2};x_0)-\bar{J}^2_1(\lambda_1u^1_{c1}+\lambda_2u^2_{c1},u^+_{c2};x_0)  \\
	&=\lambda_1\lambda_2\mathbb{E}\bigg\{ \int_0^T\left[\left\langle \bar{Q}_{21}(x_1-x_2),x_1-x_2\right\rangle +\,2 \left\langle x_1-x_2,S_{21}(u^1_{c1}-u^2_{c1})\right\rangle\right.\\
    &\qquad\qquad\qquad \left.+\left\langle R_{c1}(u^1_{c1}-u^2_{c1}),u^1_{c1}-u^2_{c1}\right\rangle\right] dt +\left\langle G_{21}(x_1-x_2)(T),(x_1-x_2)(T)\right\rangle\bigg\}.
\end{aligned}
\end{equation*}
Denote $g_1:=u^1_{c1}-u^2_{c1}$, $z'_1:=x_1-x_2$. Therefore, $z'_1(\cdot)$ is deterministic and satisfies (\ref{auxiliary J'}). Hence
$$
	\lambda_1\bar{J}^2_1(u^1_{c1},u^+_{c2};x_0)+\lambda_2\bar{J}^2_1(u^2_{c1},u^+_{c2};x_0)-\bar{J}^2_1(\lambda_1u^1_{c1}+\lambda_2u^2_{c1},u^+_{c2};x_0) =\lambda_1\lambda_2J_1'(g_1(\cdot)),
$$
and the lemma follows.
\end{proof}

\begin{Remark}
When the coefficient matrix $R_{ci}(\cdot)$ is sufficiently large, the convexity of $\bar{J}^2_1(u_{c1},u^+_{c2};\\x_0)$ and $\bar{J}^2_2(u^+_{c1},u_{c2};x_0)$ can usually be ensured.
\end{Remark}

We give the following assumptions:

{\bf (A3)} For $i=1,2$, $R_{ci}(\cdot)\gg0$.

{\bf (A4)} The map $g_i\mapsto J_i'(g_i)$ is uniformly positive definite.

For Managerial Level 2, in order to obtain the ith player's Nash equilibrium strategies, the following result can be derived through by Theorem 2.2.1 of \cite{Sun-Yong20}.

\begin{mypro}\label{prop-3.3}
	Let (A1)-(A4) hold. Suppose the attenuation index $\gamma$ is sufficiently large. Then, $(u^+_{c1}(\cdot),u^+_{c2}(\cdot))$ is an open-loop Nash equilibrium if and only if the stationary condition is satisfied:
\begin{equation*}
\left\{\begin{aligned}
	&S^\top_{21}(t)\tilde{x}(t)+R_{c1}(t)u^+_{c1}(t)+\bar{B}^\top_{c1}(t)\tilde{y}(t)+\bar{D}^\top_{c1}(t)\tilde{z}(t)=0,\quad t\in[0,T],\quad\mathbb{P}\mbox{-}a.s. \\
	&S^\top_{22}(t)\tilde{x}(t)+R_{c2}(t)u^+_{c2}(t)+\bar{B}^\top_{c2}(t)\tilde{Y}(t)+\bar{D}^\top_{c2}(t)\tilde{Z}(t)=0,\quad t\in[0,T],\quad\mathbb{P}\mbox{-}a.s.,
\end{aligned}\right.
\end{equation*}
where $(\tilde{x}(\cdot),\tilde{y}(\cdot),\tilde{z}(\cdot))$ and $(\tilde{x}(\cdot),\tilde{Y}(\cdot),\tilde{Z}(\cdot))$ satisfy the following FBSDEs respectively:
\begin{equation}\label{FBSDE1}
\left\{\begin{aligned}
	d\tilde{x}&=\big[\bar{A}\tilde{x}+\bar{B}_{c1}u^+_{c1}+\bar{B}_{c2}u^+_{c2}\big]dt+\big[\bar{C}\tilde{x}+\bar{D}_{c1}u^+_{c1}+\bar{D}_{c2}u^+_{c2}\big]dW ,\\
	d\tilde{y}&=-\big[\bar{A}^\top\tilde{y}+\bar{C}^\top\tilde{z}+\bar{Q}_{21}\tilde{x}+S_{21}u^+_{c1}\big]+\tilde{z}dW,\\
	\tilde{x}(0)&=x_0,\quad\tilde{y}(T)=G_{21}\tilde{x}(T),\\
\end{aligned}\right.
\end{equation}
\begin{equation}\label{FBSDE2}
\left\{\begin{aligned}
	d\tilde{x}&=\big[\bar{A}\tilde{x}+\bar{B}_{c1}u^+_{c1}+\bar{B}_{c2}u^+_{c2}\big]dt+\big[\bar{C}+\bar{D}_{c1}u^+_{c1}+\bar{D}_{c2}u^+_{c2}\big]dW ,\\
	d\tilde{Y}&=-\big[\bar{A}^\top\tilde{Y}+\bar{C}^\top\tilde{Z}+\bar{Q}_{22}\tilde{x}+S_{22}u^+_{c2}\big]+\tilde{Z}dW,\\
	\tilde{x}(0)&=x_0,\quad\tilde{Y}(T)=G_{22}\tilde{x}(T).
\end{aligned}\right.
\end{equation}
Moreover,  the Nash equilibrium becomes
\begin{equation}\label{Nash equilibrium2}
\left\{\begin{aligned}
    u^+_{c1}(t)&=-R^{-1}_{c1}(t)\big[S^\top_{21}(t)\tilde{x}(t)+\bar{B}^\top_{c1}(t)\tilde{y}(t)+\bar{D}^\top_{c1}(t)\tilde{z}(t)\big],\quad t\in[0,T],\quad\mathbb{P}\mbox{-}a.s. \\
	u^+_{c2}(t)&=-R^{-1}_{c2}(t)\big[S^\top_{22}(t)\tilde{x}(t)+\bar{B}^\top_{c2}(t)\tilde{Y}(t)+\bar{D}^\top_{c2}(t)\tilde{Z}(t)\big],\quad t\in[0,T],\quad\mathbb{P}\mbox{-}a.s..
\end{aligned}\right.
\end{equation}
\end{mypro}

To obtain the closed-loop representation of Nash equilibrium, plugging (\ref{Nash equilibrium2}) into (\ref{FBSDE1}) and (\ref{FBSDE2}), we have the following Hamiltonian system:
\begin{equation}\label{Hamiltonian-2}
\left\{\begin{aligned}
	d\tilde{x}=&\big[(\bar{A}-\bar{B}_{c1}R^{-1}_{c1}S^\top_{21}-\bar{B}_{c2}R^{-1}_{c2}S^\top_{22})\tilde{x}-\bar{B}_{c1}R^{-1}_{c1}\bar{B}^\top_{c1}\tilde{y}-\bar{B}_{c1}R^{-1}_{c1}\bar{D}^\top_{c1}\tilde{z}\\
    &\ -\bar{B}_{c2}R^{-1}_{c2}\bar{B}^\top_{c2}\tilde{Y}-\bar{B}_{c2}R^{-1}_{c2}\bar{D}^\top_{c2}\tilde{Z}\big]dt \\ &+\big[(\bar{C}-\bar{D}_{c1}R^{-1}_{c1}S^\top_{21}-\bar{D}_{c2}R^{-1}_{c2}S^\top_{22})\tilde{x}-\bar{D}_{c1}R^{-1}_{c1}\bar{B}^\top_{c1}\tilde{y}-\bar{D}_{c1}R^{-1}_{c1}\bar{D}^\top_{c1}\tilde{z}\\
    &\qquad -\bar{D}_{c2}R^{-1}_{c2}\bar{B}^\top_{c2}\tilde{Y}-\bar{D}_{c2}R^{-1}_{c2}\bar{D}^\top_{c2}\tilde{Z}\big]dW ,\\
	d\tilde{y}=&-\big[(\bar{A}-\bar{B}_{c1}R^{-1}_{c1}S^\top_{21})^\top\tilde{y}+(\bar{C}-\bar{D}_{c1}R^{-1}_{c1}S^\top_{21})^\top\tilde{z}+(\bar{Q}_{21}-S_{21}R^{-1}_{c1}S^\top_{21})\tilde{x}\big]+\tilde{z}dW,\\
	d\tilde{Y}=&-\big[(\bar{A}-\bar{B}_{c2}R^{-1}_{c2}S^\top_{22})^\top\tilde{Y}+(\bar{C}-\bar{D}_{c2}R^{-1}_{c2}S^\top_{22})^\top\tilde{Z}+(\bar{Q}_{22}-S_{22}R^{-1}_{c2}S^\top_{22})\tilde{x}\big]+\tilde{Z}dW,\\
	\tilde{x}(0)=&x_0,\quad\tilde{y}(T)=G_{21}\tilde{x}(T),\quad\tilde{Y}(T)=G_{22}\tilde{x}(T).
	\end{aligned}\right.
\end{equation}

Define
\begin{equation}\label{notation-2}
\begin{aligned}
	\mathbb{Y}&:=
	\begin{pmatrix}
		\tilde{y}	\\\tilde{Y}
	\end{pmatrix},\quad
	\mathbb{Z}:=
	\begin{pmatrix}
		\tilde{z}	\\\tilde{Z}
	\end{pmatrix},\quad
	\mathbb{B}_1:=
	\begin{pmatrix}
		-\bar{B}_{c1}R^{-1}_{c1}\bar{B}^\top_{c1}&-\bar{B}_{c2}R^{-1}_{c2}\bar{B}^\top_{c2}
	\end{pmatrix},\\
	\mathbb{B}_2&:=
	\begin{pmatrix}
		-\bar{B}_{c1}R^{-1}_{c1}\bar{D}^\top_{c1}&-\bar{B}_{c2}R^{-1}_{c2}\bar{D}^\top_{c2}
	\end{pmatrix},\quad
	\mathbb{D}_1:=
	\begin{pmatrix}
		-\bar{D}_{c1}R^{-1}_{c1}\bar{B}^\top_{c1}&-\bar{D}_{c2}R^{-1}_{c2}\bar{B}^\top_{c2}
	\end{pmatrix},\\
	\mathbb{D}_2&:=
	\begin{pmatrix}
		-\bar{D}_{c1}R^{-1}_{c1}\bar{D}^\top_{c1}&-\bar{D}_{c2}R^{-1}_{c2}\bar{D}^\top_{c2}
	\end{pmatrix},\quad
	\mathbb{A}:=
	\begin{pmatrix}
		\bar{A}-\bar{B}_{c1}R^{-1}_{c1}S^\top_{21}&0 \\
		0&\bar{A}-\bar{B}_{c2}R^{-1}_{c2}S^\top_{22}
	\end{pmatrix},\\
	\mathbb{C}&:=
	\begin{pmatrix}
		\bar{C}-\bar{D}_{c1}R^{-1}_{c1}S^\top_{21}&0 \\
		0&\bar{C}-\bar{D}_{c2}R^{-1}_{c2}S^\top_{22}
	\end{pmatrix},\quad
	\mathbb{Q}:=
	\begin{pmatrix}
		\bar{Q}_{21}-S_{21}R^{-1}_{c1}S^\top_{21}\\\bar{Q}_{22}-S_{22}R^{-1}_{c2}S^\top_{22}
	\end{pmatrix},\quad
	\mathbb{G}:=
	\begin{pmatrix}
		G_{21}\\G_{22}
	\end{pmatrix}.
\end{aligned}
\end{equation}
Then (\ref{Hamiltonian-2}) can be rewrite as follows:
\begin{equation}
\left\{\begin{aligned}
    d\tilde{x}&=\big[(\bar{A}-\bar{B}_{c1}R^{-1}_{c1}S^\top_{21}-\bar{B}_{c2}R^{-1}_{c2}S^\top_{22})\tilde{x}+\mathbb{B}_1\mathbb{Y}+\mathbb{B}_2\mathbb{Z}\big]dt \\
    &\quad +\big[(\bar{C}-\bar{D}_{c1}R^{-1}_{c1}S^\top_{21}-\bar{D}_{c2}R^{-1}_{c2}S^\top_{22})\tilde{x}+\mathbb{D}_1\mathbb{Y}+\mathbb{D}_2\mathbb{Z}\big]dW, \\
    d\mathbb{Y}&=-[\mathbb{A}^\top\mathbb{Y}+\mathbb{C}^\top\mathbb{Z}+\mathbb{Q}\tilde{x}]dt+\mathbb{Z}dW,\\
    \tilde{x}(0)&=x_0,\quad \mathbb{Y}(T)=\mathbb{G}\tilde{x}(T).
\end{aligned}\right.
\end{equation}

\begin{mypro}
Suppose that the following equation
\begin{equation}\label{Pi equation}
\left\{\begin{aligned}
	&\dot{\Pi}+\Pi\bigg(\bar{A}-\sum_{i=1}^2\bar{B}_{ci}R^{-1}_{ci}S^\top_{2i}\bigg)+\Pi\mathbb{B}_1\Pi+\mathbb{A}^\top\Pi+\mathbb{Q}+(\mathbb{C}^\top+\Pi\mathbb{B}_2)(I-\Pi\mathbb{D}_2)^{-1} \\
	&\ \times\bigg(\Pi\bar{C}-\sum_{i=1}^2\Pi\bar{D}_{ci}R^{-1}_{ci}S^\top_{2i}+\Pi\mathbb{D}_1\Pi\bigg)=0,\\
	&\Pi(T)=\mathbb{G},
\end{aligned}\right.
\end{equation}
admits a solution $\Pi(\cdot)$ such that $I-\Pi\mathbb{D}_2$ has bounded inverse, then the Nash equilibrium (\ref{Nash equilibrium2}) has the following representation:
\begin{equation}\label{Nash equilibrium2 sfb}
\begin{aligned}
	u^+_{c1}(t)&=-R^{-1}_{c1}\bigg[S^\top_{21}+\left (\bar{B}^\top_{c1}\quad0\right)\Pi+\left (\bar{D}^\top_{c1}\quad0\right)(I-\Pi\mathbb{D}_2)^{-1}\\
    &\qquad\qquad \times\bigg(\Pi\bar{C}-\sum_{i=1}^2\Pi\bar{D}_{ci}R^{-1}_{ci}S^\top_{2i}+\Pi\mathbb{D}_1\Pi\bigg)\bigg]\tilde{x}(t),\\
	u^+_{c2}(t)&=-R^{-1}_{c2}\bigg[S^\top_{22}+\left (0\quad\bar{B}^\top_{c2}\right)\Pi+\left (0\quad\bar{D}^\top_{c2}\right)(I-\Pi\mathbb{D}_2)^{-1}\\
    &\qquad\qquad \times\bigg(\Pi\bar{C}-\sum_{i=1}^2\Pi\bar{D}_{ci}R^{-1}_{ci}S^\top_{2i}+\Pi\mathbb{D}_1\Pi\bigg)\bigg]\tilde{x}(t),
\end{aligned}
\end{equation}
where $\tilde{x}(\cdot)$ satisfies
\begin{equation}
\left\{\begin{aligned}
    d\tilde{x}&=\bigg[\bar{A}-\bar{B}_{c1}R^{-1}_{c1}S^\top_{21}-\bar{B}_{c2}R^{-1}_{c2}S^\top_{22}+\mathbb{B}_1\Pi\\
    &\qquad +\mathbb{B}_2(I-\Pi\mathbb{D}_2)^{-1}\bigg(\Pi\bar{C}-\sum_{i=1}^2\Pi\bar{D}_{ci}R^{-1}_{ci}S^\top_{2i}+\Pi\mathbb{D}_1\Pi\bigg)\bigg]\tilde{x}dt \\	
    &\quad +\bigg[\bar{C}-\bar{D}_{c1}R^{-1}_{c1}S^\top_{21}-\bar{D}_{c2}R^{-1}_{c2}S^\top_{22}+\mathbb{D}_1\Pi\\
    &\qquad\quad +\mathbb{D}_2(I-\Pi\mathbb{D}_2)^{-1}(\Pi\bar{C}-\sum_{i=1}^2\Pi\bar{D}_{ci}R^{-1}_{ci}S^\top_{2i}+\Pi\mathbb{D}_1\Pi)\big]\tilde{x}dW, \\
	\tilde{x}(0)&=x_0.
\end{aligned}\right.
\end{equation}
\end{mypro}

\begin{proof}
Let $\mathbb{Y}=\Pi\tilde{x}$ with $\Pi=\begin{pmatrix}
		\Pi_1\\\Pi_2
	\end{pmatrix}$, we have
\begin{equation*}
	\mathbb{Z}=(I-\Pi\mathbb{D}_2)^{-1}\bigg(\Pi\bar{C}-\sum_{i=1}^{2}\Pi\bar{D}_{ci}R^{-1}_{ci}S^\top_{2i}+\Pi\mathbb{D}_1\Pi\bigg)\tilde{x}=\Sigma\tilde{x},
\end{equation*}
with  $\Sigma=\begin{pmatrix}
	\Sigma_1\\\Sigma_2
\end{pmatrix}$, and $\Pi$ satisfies (\ref{Pi equation}), the proof is similar to that of  Proposition \ref{prop-3.2}, we omit the details here.
\end{proof}
For (\ref{Gamma 1i}) to become the incentive strategy of Decision-making Level 1, (\ref{Nash equilibrium2 sfb}) must be matched with
the corresponding team-optimal strategies of Managerial Level 2 and Executive Level 3 from some conditions. We assume that the following equation holds:
\begin{equation}\label{equal 1}
	u^+_{ci}(t)=u^*_{ci}(t),\qquad t\in[0,T].
\end{equation}

\begin{Remark}\label{remark-3.4}
When the relation (\ref{equal 1}) holds, we will verify that $\tilde{x}(t)=\bar{x}(t)$. By (\ref{Hamiltonian-1}),
\begin{equation}
\begin{aligned}\label{bar x'}
	d\bar{x}&=[A\bar{x}+B_{c}u^*_{c}+Ev^*]dt+[C\bar{x}+D_{c}u^*_{c}]dW \\
	&=\bigg[(A+\gamma^{-2}EE^\top P)\bar{x}+\sum_{i=1}^2 B_{1i}u^*_{1i}+\sum_{i=1}^2 \sum_{j=1}^3 B_{2ij}u^*_{2ij}+\sum_{j=1}^{3}\sum_{i=1}^{2}B_{3ji}u^*_{3ji}\bigg]dt\\
	&\quad +\bigg[ C\bar{x}+\sum_{i=1}^2 D_{1i}u^*_{1i}+\sum_{i=1}^2 \sum_{j=1}^3 D_{2ij}u^*_{2ij}+\sum_{j=1}^3\sum_{i=1}^2 D_{3ji}u^*_{3ji}\bigg]dW.
\end{aligned}
\end{equation}
Plugging (\ref{u*_1i}) into (\ref{bar x'}), we get
\begin{equation}
\left\{\begin{aligned}
 	d\bar{x}&=\bigg[\bigg(A+\gamma^{-2}EE^\top P+\sum_{i=1}^2\sum_{j=1}^3 B_{1i}\xi_{1ij}\bigg)\bar{x}+\sum_{i=1}^2\sum_{j=1}^3(B_{2ij}+B_{1i}\eta_{1ij})u^*_{2ij}\\
    &\qquad +\sum_{j=1}^3\sum_{i=1}^2(B_{3ji}+B_{1i}\zeta_{1ij})u^*_{3ji}\bigg]dt \\
 	&\quad +\bigg[\bigg(C+\sum_{i=1}^{2}\sum_{j=1}^3D_{1i}\xi_{1ij}\bigg)\bar{x}+\sum_{i=1}^2\sum_{j=1}^3(D_{2ij}+D_{1i}\eta_{1ij})u^*_{2ij}\\
    &\qquad\quad +\sum_{j=1}^3\sum_{i=1}^2(D_{3ji}+D_{1i}\zeta_{1ij})u^*_{3ji}\bigg]dW,\\
 	\bar{x}(0)&=x_0.
\end{aligned}\right.
\end{equation}
By Proposition \ref{prop-3.3} and notations (\ref{notation-1}), we have
\begin{equation}
\left\{\begin{aligned}
	d\tilde{x}&=\bigg[\bigg(A+\gamma^{-2}EE^\top P+\sum_{i=1}^2\sum_{j=1}^3B_{1i}\xi_{1ij}\bigg)\tilde{x}+\sum_{i=1}^2\sum_{j=1}^3(B_{2ij}+B_{1i}\eta_{1ij})u^+_{2ij}\\
    &\qquad +\sum_{j=1}^3\sum_{i=1}^2(B_{3ji}+B_{1i}\zeta_{1ij})u^+_{3ji}\bigg]dt \\
	&\quad +\bigg[\bigg(C+\sum_{i=1}^2\sum_{j=1}^3D_{1i}\xi_{1ij}\bigg)\tilde{x}+\sum_{i=1}^2\sum_{j=1}^3(D_{2ij}+D_{1i}\eta_{1ij})u^+_{2ij}\\
    &\qquad\quad +\sum_{j=1}^3\sum_{i=1}^{2}(D_{3ji}+D_{1i}\zeta_{1ij})u^+_{3ji}\bigg]dW,\\
	\tilde{x}(0)&=x_0,\\
\end{aligned}\right.
\end{equation}
then the subtraction of these two equations yields the following equation:
\begin{equation}
\left\{\begin{aligned}
	d(\tilde{x}-\bar{x})&=\bigg(A+\gamma^{-2}EE^\top P+\sum_{i=1}^2\sum_{j=1}^3B_{1i}\xi_{1ij}\bigg)(\tilde{x}-\bar{x})dt\\
    &\quad +\bigg(C\sum_{i=1}^{2}\sum_{j=1}^{3}D_{1i}\xi_{1ij}\bigg)(\tilde{x}-\bar{x})dW,\\
	(\tilde{x}-\bar{x})(0)&=0.
\end{aligned}\right.
\end{equation}
Obviously, $\tilde{x}-\bar{x}\equiv0$ is the solution to the above equation. Therefore,
$$\tilde{x}(t)=\bar{x}(t),\qquad t\in[0,T],\qquad \mathbb{P}\mbox{-}a.s..$$
\end{Remark}

When the relation (\ref{equal 1}) holds, it can be expressed briefly as follows:
\begin{equation}\label{equality 1}
     \bigg[R^{-1}_{ci}\big(S^\top_{2i}+\bar{B}^\top_{ci}\Pi_i+\bar{D}^\top_{ci}\Sigma_i\big)-R^{-1}_{1i}\big(B^\top_{i}P+D^\top_{i}\Lambda\big)\bigg]\bar{x}(t)=0,\quad i=1,2,\; t\in[0,T],\;\mathbb{P}\mbox{-}a.s..
\end{equation}

\begin{Remark}
In fact, equation (\ref{equality 1}) contains twelve equations, in which $\eta_{1ij}$, $\zeta_{1ij}$ and the solution $\Pi(\cdot)$ of equation (\ref{Pi equation}) are mutually coupled.
\end{Remark}

Summarizing what is stated above, we obtain the theorem with the incentive strategy of Decision-making Level 1 under the additional condition.
\begin{mythm}\label{thm-3.1}
	Let (A1)-(A4) hold, suppose the attenuation $\gamma$ is sufficiently large. If there exist matrix-valued functions $\eta^*_{1ij}$ and $\zeta^*_{1ij}$ such that (\ref{Pi equation}) admits solution $\Pi^*(\cdot)$, satisfying equation (\ref{equality 1}), then there exists the incentive strategy set of Decision-making Level 1 in the three-level incentive Stackelberg game under the $H_\infty$ constraint with multiple leaders and multiple followers (\ref{state})-(\ref{cost_3j}), given by
\begin{equation}\label{Gamma*1i}
\begin{aligned}
	&\Gamma^*_{1i}\big(x(t),u_{2i}(t),u_{31i}(t),u_{32i}(t),u_{33i}(t),t\big)\\
    &=\sum_{j=1}^3\xi^*_{1ij}(t)x(t)+\sum_{j=1}^3\eta^*_{1ij}(t)u_{2ij}(t)+\sum_{j=1}^3\zeta^*_{1ij}(t)u_{3ji}(t),
\end{aligned}
\end{equation}
where
\begin{equation*}
\begin{aligned}
\xi^*_{1ij}&:=-R^{-1}_i(B^\top_{1i}P+D^\top_{1i}\Lambda)+\eta^*_{1ij}R^{-1}_{1ij}(B^\top_{2ij}P+D^\top_{2ij}\Lambda)+\zeta^*_{1ij}\bar{R}^{-1}_{1ji}(B^\top_{3ji}P+D^\top_{3ji}\Lambda),\\
\Lambda&:=(I+PD_cR^{-1}_cD^\top_c)^{-1}P(C-D_cR^{-1}_cB^\top_cP).
\end{aligned}
\end{equation*}
\end{mythm}

\begin{proof}
The proof is obvious from what is stated prior to Theorem \ref{thm-3.1}, we omit it here.
\end{proof}

\begin{Remark}
We notice that the following equation holds, which will be used next:
\begin{equation}\label{u*1i equality}
\begin{aligned}
	u^*_{1i}(t)&=\Gamma^*_{1i}(\bar{x}(t),u^*_{2i}(t),u^*_{31i}(t),u^*_{32i}(t),u^*_{33i}(t),t) \\
	&=\sum_{j=1}^3\xi^*_{1ij}(t)\bar{x}(t)+\sum_{j=1}^3\eta^*_{1ij}(t)u^*_{2ij}(t)+\sum_{j=1}^3\zeta^*_{1ij}(t)u^*_{3ji}(t).
\end{aligned}
\end{equation}
$\xi^*_{1ij}$, $\eta^*_{1ij}$ and $\zeta^*_{1ij}$ depend on the initial state value $x_0$, because the equation (\ref{equality 1}) includes the team-optimal trajectory $\bar{x}(\cdot)$, which depends on $x_0$.
\end{Remark}

\subsection{Nash equilibrium of Executive Level 3}

Since a sufficient condition for the incentive strategy for Decision-making Level 1 has been obtained by Theorem \ref{thm-3.1}, we will derive one for Managerial Level 2 next. Decision-making Level 1 announces the incentive strategy (\ref{Gamma*1i}) and Managerial Level 2 announces the incentive strategy ahead of time, three people in Executive Level 3 determine their strategies to achieve a Nash equilibrium by
responding to the announced strategy of Decision-making Level 1 and Managerial Level 2.

We consider the following announced strategy of the ith of Managerial Level 2:
\begin{equation}\label{Gamma 2ij}
\begin{aligned}
	u_{2ij}(t)&=\Gamma_{2ij}\big(x(t),u_{3ji}(t),t\big)\\
	&=-R^{-1}_{1ij}(B^\top_{2ij}P+D^\top_{2ij}\Lambda)x+\rho_{2ij}\big[u_{3ji}+\bar{R}^{-1}_{1ji}(B^\top_{3ji}P+D^\top_{3ji}\Lambda)x\big]\\
	&=\big[-R^{-1}_{1ij}(B^\top_{2ij}P+D^\top_{2ij}\Lambda)+\rho_{2ij}\bar{R}^{-1}_{1ji}(B^\top_{3ji}P+D^\top_{3ji}\Lambda)\big]x+\rho_{2ij}u_{3ji}\\
	&=\theta_{2ij}(t)x(t)+\rho_{2ij}(t)u_{3ji}(t),\\
\end{aligned}
\end{equation}
where
\[\theta_{2ij}:=-R^{-1}_{1ij}(B^\top_{2ij}P+D^\top_{2ij}\Lambda)+\rho_{2ij}\bar{R}^{-1}_{1ji}(B^\top_{3ji}P+D^\top_{3ji}\Lambda),\]
and $\theta_{2ij}(t)\in\mathbb{R}^{m_{2ij}\times n}$, $\rho_{2ij}(t)\in\mathbb{R}^{m_{2ij}\times m_{3ji}}$ are (unknown) strategy parameter matrices whose components are continuous functions of $t$ on the interval $[0,T]$ respectively. Notice that
\begin{equation}\label{u*_2ij}
		u^*_{2ij}(t)=\Gamma_{1i}(\bar{x}(t),u^*_{3ji}(t),t)=\theta_{2ij}(t)\bar{x}(t)+\rho_{2ij}(t)u^*_{3ji}(t).
\end{equation}

Substituting (\ref{Gamma*1i}), (\ref{Gamma 2ij}) and (\ref{v}) into state equation (\ref{state}) and cost functional for the jth member of Executive Level 3 (\ref{cost_3j}), we will get
\begin{equation}
\left\{\begin{aligned}
	dx&=\bigg[\hat{A}x+\sum_{j=1}^3\sum_{i=1}^2\hat{B}_{3ji}u_{3ji}\bigg]dt+\bigg[\hat{C}x+\sum_{j=1}^3\sum_{i=1}^2\hat{D}_{3ji}u_{3ji}\bigg]dW\\
	&=\bigg[\hat{A}x+\sum_{j=1}^3\hat{B}_{3j}u_{3j}\bigg]dt+\bigg[\hat{C}x+\sum_{j=1}^3\hat{D}_{3j}u_{3j}\bigg]dW,\\
	x(0)&=x_0,
\end{aligned}\right.
\end{equation}
\begin{equation}
\begin{aligned}
	\bar{J}^3_j(u_{31},u_{32},u_{33};x_0)&=\mathbb{E}\bigg\{ \int_{0}^{T}\big[\langle  \hat{Q}_{3j}(t)x(t),x(t)\rangle  +\,2 \left\langle x(t),S_{3j}(t)u_{3j}(t)\right\rangle\\
    &\qquad\qquad +\left\langle R_{3j}(t)u_{3j}(t),u_{3j}(t)\right\rangle\big]dt +\left\langle G_{3j}x(T),x(T)\right\rangle \bigg\} ,
\end{aligned}
\end{equation}
where
\begin{equation*}
\begin{aligned}
	\hat{A}&:=A+\gamma^{-2}EE^\top P+\sum_{i=1}^2\sum_{j=1}^3B_{1i}(\xi^*_{1ij}+\eta^*_{1ij}\theta_{2ij})+\sum_{i=1}^2\sum_{j=1}^3B_{2ij}\theta_{2ij},\\
	\hat{C}&:=C+\sum_{i=1}^2\sum_{j=1}^3D_{1i}(\xi^*_{1ij}+\eta^*_{1ij}\theta_{2ij})+\sum_{i=1}^2\sum_{j=1}^3D_{2ij}\theta_{2ij},\quad u_{3j}(t):=\mathbf{col}\big[u_{3j1}(t)\;u_{3j2}(t)\big],\\
	\hat{B}_{3ji}&:=B_{1i}(\eta^*_{1ij}\rho_{2ij}+\zeta^*_{1ij})+B_{2ij}\rho_{2ij}+B_{3ji},\quad \hat{D}_{3ji}:=D_{1i}(\eta^*_{1ij}\rho_{2ij}+\zeta^*_{1ij})+D_{2ij}\rho_{2ij}+D_{3ji},\\
	\hat{B}_{3j}&:=\begin{pmatrix}
		\hat{B}_{3j1}&\hat{B}_{3j2}
	\end{pmatrix},\quad
	\hat{D}_{3j}:=\begin{pmatrix}
		\hat{D}_{3j1}&\hat{D}_{3j2}
	\end{pmatrix},\quad
	\hat{Q}_{3j}:=Q_{3j}+\sum_{i=1}^{2}\theta^\top_{2ij}R_{3ij}\theta_{2ij},\\
	S_{3j}&:=\begin{pmatrix}
		\theta^\top_{21j}R_{31j}\rho_{21j}&\theta^\top_{22j}R_{32j}\rho_{22j}
	\end{pmatrix},\quad
	R_{3j}:=\begin{pmatrix}
		\bar{R}_{3j1}+\rho^\top_{21j}R_{31j}\rho_{21j}&0 \\
		0&\bar{R}_{3j2}+\rho^\top_{22j}R_{32j}\rho_{22j}
	\end{pmatrix}.
\end{aligned}
\end{equation*}
Next, Executive Level 3 need to find the Nash equilibrium $(\bar{u}_{31}(\cdot),\bar{u}_{32}(\cdot),\bar{u}_{33}(\cdot))\in \mathcal{U}_{31}\times \mathcal{U}_{32}\times \mathcal{U}_{33}$, such that
\begin{equation}
\begin{aligned}
	\bar{J}^3_1(\bar{u}_{31},\bar{u}_{32},\bar{u}_{33};x_0)\leq\bar{J}^3_1(u_{31},\bar{u}_{32},\bar{u}_{33};x_0),\\
	\bar{J}^3_2(\bar{u}_{31},\bar{u}_{32},\bar{u}_{33};x_0)\leq\bar{J}^3_2(\bar{u}_{31},u_{32},\bar{u}_{33};x_0),\\
	\bar{J}^3_3(\bar{u}_{31},\bar{u}_{32},\bar{u}_{33};x_0)\leq\bar{J}^3_3(\bar{u}_{31},\bar{u}_{32},u_{33};x_0),\\
\end{aligned}
\end{equation}
i.e., for $j=1,2,3$,
\[\bar{J}^3_j(\bar{u}_{31},\bar{u}_{32},\bar{u}_{33};x_0)\leq\bar{J}^3_j(\bar{u}_{-3j}(u_{3j});x_0),\]
for any $(u_{31}(\cdot),u_{32}(\cdot),u_{33}(\cdot))\in (\bar{u}_{31}(\cdot),\bar{u}_{32}(\cdot),\bar{u}_{33}(\cdot))\in \mathcal{U}_{31}\times \mathcal{U}_{32}\times \mathcal{U}_{33}$.

\begin{mylem}\label{lemma3.2}
Let (A1)-(A4) hold. For any $x_0\in\mathbb{R}^n$, $\bar{J}^3_j(\bar{u}_{-3j}(u_{3j});x_0)$ is convex (resp., strictly convex) in $u_{3j}(\cdot)\in L_{\mathbb{F}}^2(0,T;\mathbb{R}^{\sum_{i=1}^{2}m_{3ji}})$ if and only if $J_j''(h_j(\cdot))$ is positive semidefinite (resp., positive definite), where
$$
	J_j''(h_j(\cdot)):=\mathbb{E}\left\lbrace \int_0^T\left[\big\langle \hat{Q}_{3j}z''_j,z''_j\big\rangle +\,2 \left\langle z''_j,S_{3j}h_j\right\rangle+\left\langle R_{3j}h_j,h_j\right\rangle\right]  dt
    +\left\langle G_{3j}z''_j(T),z''_j(T)\right\rangle \right\rbrace,
$$
and $z''_j(\cdot)$ satisfies:
\begin{equation}\label{auxiliary J''}
\left\{\begin{aligned}
	dz''_j(t)&=[\hat{A}z''_j+\hat{B}_{3j}h_j]dt+[\hat{C}z''_j+\hat{D}_{3j}h_j]dW ,\\
	z''_j(0)&=0.
\end{aligned}\right.
\end{equation}
\end{mylem}

\begin{proof}
Similar to the proof of Lemma \ref{lemma3.1}, we omit it.
\end{proof}

\begin{Remark}
When the coefficient matrix $R_{3j}(\cdot)$ is sufficiently large, the convexity of \\$\bar{J}^3_j(\bar{u}_{-3j}(u_{3j});x_0)$ can usually be ensured.
\end{Remark}

We introduce the following assumption:

{\bf (A5)} The map $h_j\mapsto J_j''(h_j)$ is uniformly positive definite.

\begin{mypro}
Let (A1)-(A5) hold. Suppose the attenuation index $\gamma$ is sufficiently large. Then, $(\bar{u}_{31}(\cdot),\bar{u}_{32}(\cdot),\bar{u}_{33}(\cdot))$ is an open-loop Nash equilibrium if and only if stationary condition is satisfied:
\begin{equation*}
	R_{3j}(t)\bar{u}_{3j}(t)+S^\top_{3j}(t)\hat{x}(t)+\hat{B}^\top_{3j}(t)\hat{y}_j(t)+\hat{D}^\top_{3j}(t)\hat{z}_j(t)=0,\quad t\in[0,T],\quad\mathbb{P}\mbox{-}a.s.,
\end{equation*}
where $(\hat{x}(\cdot),\hat{y}_j(\cdot),\hat{z}_j(\cdot))$ satisfies the following FBSDE:
\begin{equation}\label{FBSDEj}
\left\{\begin{aligned}
	d\hat{x}&=\bigg[\hat{A}\hat{x}+\sum_{j=1}^3\hat{B}_{3j}\bar{u}_{3j}\bigg]dt+\bigg[\hat{C}\hat{x}+\sum_{j=1}^3\hat{D}_{3j}\bar{u}_{3j}\bigg]dW ,\\
	d\hat{y}_j&=-\bigg[\hat{A}^\top\hat{y}_j+\hat{C}^\top\hat{z}_j+\hat{Q}_{3j}\hat{x}+S_{3j}\bar{u}_{3j}\bigg]+\hat{z}_jdW,\\
	\hat{x}(0)&=x_0,\quad\hat{y}_j(T)=G_{3j}\hat{x}(T).
\end{aligned}\right.
\end{equation}
Moreover, the Nash equilibrium becomes
\begin{equation}\label{Nash equilibrium3}
	\bar{u}_{3j}(t)=-R^{-1}_{3j}(t)\big[S^\top_{3j}(t)\hat{x}(t)+\hat{B}^\top_{3j}(t)\hat{y}_{j}(t)+\hat{D}^\top_{3j}(t)\hat{z}_{j}(t)\big],\quad t\in[0,T],\quad\mathbb{P}\mbox{-}a.s..
\end{equation}
\end{mypro}

\begin{proof}
The details are omitted here.
\end{proof}

To obtain the state feedback representation of Nash equilibrium, plugging (\ref{Nash equilibrium3}) into  (\ref{FBSDEj}), we have the following Hamiltonian system:
\begin{equation}\label{Hamiltonian 3}
\left\{\begin{aligned}
	d\hat{x}&=\bigg[\bigg(\hat{A}-\sum_{j=1}^3\hat{B}_{3j}R^{-1}_{3j}S^\top_{3j}\bigg)\hat{x}-\sum_{j=1}^3\hat{B}_{3j}R^{-1}_{3j}\hat{B}^\top_{3j}\hat{y}_j-\sum_{j=1}^3\hat{B}_{3j}R^{-1}_{3j}\hat{D}^\top_{3j}\hat{z}_j\bigg]dt \\
	&\quad +\bigg[\big(\hat{C}-\sum_{j=1}^3\hat{D}_{3j}R^{-1}_{3j}S^\top_{3j}\big)\hat{x}-\sum_{j=1}^3\hat{D}_{3j}R^{-1}_{3j}\hat{B}^\top_{3j}\hat{y}_j-\sum_{j=1}^3\hat{D}_{3j}R^{-1}_{3j}\hat{D}^\top_{3j}\hat{z}_j\bigg]dW,\\
	d\hat{y}_1&=-\Big[(\hat{A}-\hat{B}_{31}R^{-1}_{31}S^\top_{31})^\top\hat{y}_1+(\hat{C}-\hat{D}_{31}R^{-1}_{31}S^\top_{31})^\top\hat{z}_1+(\hat{Q}_{31}-S_{31}R^{-1}_{31}S^\top_{31})\Big]dt+\hat{z}_1dW,\\
	d\hat{y}_2&=-\Big[(\hat{A}-\hat{B}_{32}R^{-1}_{32}S^\top_{32})^\top\hat{y}_2+(\hat{C}-\hat{D}_{32}R^{-1}_{32}S^\top_{32})^\top\hat{z}_2+(\hat{Q}_{32}-S_{32}R^{-1}_{32}S^\top_{32})\Big]dt+\hat{z}_2dW,\\
	d\hat{y}_3&=-\Big[(\hat{A}-\hat{B}_{33}R^{-1}_{33}S^\top_{33})^\top\hat{y}_3+(\hat{C}-\hat{D}_{33}R^{-1}_{33}S^\top_{33})^\top\hat{z}_3+(\hat{Q}_{33}-S_{33}R^{-1}_{33}S^\top_{33})\Big]dt+\hat{z}_3dW,\\
	\hat{x}(0)&=x_0,\quad\hat{y}_1(T)=G_{31}\hat{x}(T),\quad\hat{y}_2(T)=G_{32}\hat{x}(T),\quad\hat{y}_3(T)=G_{33}\hat{x}(T).
\end{aligned}\right.
\end{equation}

Set
\begin{equation*}
\begin{aligned}
	\mathbf{Y}&:=\begin{pmatrix}
		\hat{y}_1\\
		\hat{y}_2\\
		\hat{y}_3
	\end{pmatrix},\quad
	\mathbf{Z}:=\begin{pmatrix}
		\hat{z}_1\\
		\hat{z}_2\\
		\hat{z}_3
	\end{pmatrix},\quad
	\mathbf{B}_1:=\begin{pmatrix}
		-\hat{B}_{31}R^{-1}_{31}\hat{B}^\top_{31}& -\hat{B}_{32}R^{-1}_{32}\hat{B}^\top_{32} & -\hat{B}_{33}R^{-1}_{33}\hat{B}^\top_{33}
	\end{pmatrix},\\
	\mathbf{B}_2&:=\begin{pmatrix}
		-\hat{B}_{31}R^{-1}_{31}\hat{D}^\top_{31}& -\hat{B}_{32}R^{-1}_{32}\hat{D}^\top_{32} & -\hat{B}_{33}R^{-1}_{33}\hat{D}^\top_{33}
	\end{pmatrix},\\
	\mathbf{D}_1&:=\begin{pmatrix}
		-\hat{D}_{31}R^{-1}_{31}\hat{B}^\top_{31}& -\hat{D}_{32}R^{-1}_{32}\hat{B}^\top_{32} & -\hat{D}_{33}R^{-1}_{33}\hat{B}^\top_{33}
	\end{pmatrix},\\
	\mathbf{D}_2&:=\begin{pmatrix}
		-\hat{D}_{31}R^{-1}_{31}\hat{D}^\top_{31}& -\hat{D}_{32}R^{-1}_{32}\hat{D}^\top_{32} & -\hat{D}_{33}R^{-1}_{33}\hat{D}^\top_{33}
	\end{pmatrix},\\
\end{aligned}
\end{equation*}
\begin{equation*}
\begin{aligned}
	\mathbf{A}&:=\begin{pmatrix}
		\hat{A}-\hat{B}_{31}R^{-1}_{31}S^\top_{31}&0  &0\\
		0& \hat{A}-\hat{B}_{32}R^{-1}_{32}S^\top_{32} &0 \\
		0& 0 &\hat{A}-\hat{B}_{33}R^{-1}_{33}S^\top_{33}
	\end{pmatrix},\\
	\mathbf{C}&:=\begin{pmatrix}
		\hat{C}-\hat{D}_{31}R^{-1}_{31}S^\top_{31}&  0& 0\\
		0& \hat{C}-\hat{D}_{32}R^{-1}_{32}S^\top_{32} &0 \\
		0& 0 &\hat{C}-\hat{D}_{33}R^{-1}_{33}S^\top_{33}
	\end{pmatrix},\\
	\mathbf{Q}&:=\begin{pmatrix}
		\hat{Q}_{31}-S_{31}R^{-1}_{31}S^\top_{31}\\
		\hat{Q}_{32}-S_{32}R^{-1}_{32}S^\top_{32}\\
		\hat{Q}_{33}-S_{33}R^{-1}_{33}S^\top_{33}
	\end{pmatrix},\quad
	\mathbf{G}:=\begin{pmatrix}
		G_{31}\\
		G_{32}\\
		G_{33}
	\end{pmatrix}.
\end{aligned}
\end{equation*}
We can express the Hamiltonian system (\ref{Hamiltonian 3}) more compactly as follows:
\begin{equation}\label{Hamiltonian 3'}
\left\{\begin{aligned}
    d\hat{x}&=\bigg[\bigg(\hat{A}-\sum_{j=1}^3\hat{B}_{3j}R^{-1}_{3j}S^\top_{3j}\bigg)\hat{x}+\mathbf{B}_1\mathbf{Y}+\mathbf{B}_2\mathbf{Z}\bigg]dt\\
    &\quad +\bigg[\bigg(\hat{C}-\sum_{j=1}^3\hat{D}_{3j}R^{-1}_{3j}S^\top_{3j}\bigg)\hat{x}+\mathbf{D}_1\mathbf{Y}+\mathbf{D}_2\mathbf{Z}\bigg]dW,\\
	d\mathbf{Y}&=-\big[\mathbf{A}^\top\mathbf{Y}+\mathbf{C}^\top\mathbf{Z}+\mathbf{Q}\hat{x}\big]dt+\mathbf{Z}dW,\\
	\hat{x}(0)&=x_0,\quad\mathbf{Y}(T)=\mathbf{G}\hat{x}(T).
\end{aligned}\right.
\end{equation}

\begin{mypro}
Suppose that the following equation
\begin{equation}\label{Phi}
\left\{\begin{aligned}
	&\dot{\Phi}+\Phi\bigg(\hat{A}-\sum_{j=1}^3\hat{B}_{3j}R^{-1}_{3j}S^\top_{3j}\bigg)+\mathbf{A}^\top\Phi+\Phi\mathbf{B}_1\Phi+\mathbf{Q}+(\mathbf{C}^\top+\Phi\mathbf{B}_2)(I-\Phi\mathbf{D}_2)^{-1}\\
	&\quad\times\bigg(\Phi\hat{C}-\sum_{j=1}^3\Phi\hat{D}_{3j}R^{-1}_{3j}S^\top_{3j}+\Phi\mathbf{D}_1\Phi\bigg)=0,\\
	&\Phi(T)=\mathbf{G},
\end{aligned}\right.
\end{equation}
admits a solution $\Phi(\cdot)$ such that $I-\Phi\mathbf{D}_2$ has bounded inverse, then the Nash equilibrium (\ref{Nash equilibrium3}) has the following representation:
\begin{equation*}
\begin{aligned}
	\bar{u}_{31}(t)&=-R^{-1}_{31}\bigg[S^\top_{31}+\begin{pmatrix}
		\hat{B}^\top_{31}& 0 &0
	\end{pmatrix}\Phi+\begin{pmatrix}
	\hat{D}^\top_{31}& 0 &0
	\end{pmatrix}(I-\Phi\mathbf{D}_2)^{-1}\\
    &\qquad\qquad \times\bigg(\Phi\hat{C}-\sum_{j=1}^3\Phi\hat{D}_{3j}R^{-1}_{3j}S^\top_{3j}+\Phi\mathbf{D}_1\Phi\bigg)\bigg]\hat{x}(t),\\
	\bar{u}_{32}(t)&=-R^{-1}_{32}\bigg[S^\top_{32}+\begin{pmatrix}
		0& \hat{B}^\top_{32} &0
	\end{pmatrix}\Phi+\begin{pmatrix}
		0& \hat{D}^\top_{32} &0
	\end{pmatrix}(I-\Phi\mathbf{D}_2)^{-1}\\
    &\qquad\qquad \times\bigg(\Phi\hat{C}-\sum_{j=1}^3\Phi\hat{D}_{3j}R^{-1}_{3j}S^\top_{3j}+\Phi\mathbf{D}_1\Phi\bigg)\bigg]\hat{x}(t),\\
\end{aligned}
\end{equation*}
\begin{equation}\label{bar u3j}
\begin{aligned}
	\bar{u}_{33}(t)&=-R^{-1}_{33}\bigg[S^\top_{33}+\begin{pmatrix}
		0& 0 &\hat{B}^\top_{33}
	\end{pmatrix}\Phi+\begin{pmatrix}
		0& 0 &\hat{D}^\top_{33}
	\end{pmatrix}(I-\Phi\mathbf{D}_2)^{-1}\\
    &\qquad\qquad \times\bigg(\Phi\hat{C}-\sum_{j=1}^3\Phi\hat{D}_{3j}R^{-1}_{3j}S^\top_{3j}+\Phi\mathbf{D}_1\Phi\bigg)\bigg]\hat{x}(t),
\end{aligned}
\end{equation}
where $\hat{x}(\cdot)$ satisfies
\begin{equation}\label{hat x}
\left\{\begin{aligned}
    d\hat{x}&=\bigg[\hat{A}-\sum_{j=1}^3\hat{B}_{3j}R^{-1}_{3j}S^\top_{3j}+\mathbf{B}_1\Phi+\mathbf{B}_2(I-\Phi\mathbf{D}_2)^{-1}\\
    &\qquad \times\bigg(\Phi\hat{C}-\sum_{j=1}^3\Phi\hat{D}_{3j}R^{-1}_{3j}S^\top_{3j}+\Phi\mathbf{D}_1\Phi\bigg)\bigg]\hat{x}dt\\
	&\quad +\bigg[\hat{C}-\sum_{j=1}^3\hat{D}_{3j}R^{-1}_{3j}S^\top_{3j}+\mathbf{D}_1\Phi+\mathbf{D}_2(I-\Phi\mathbf{D}_2)^{-1}\\
    &\qquad\quad \times\bigg(\Phi\hat{C}-\sum_{j=1}^3\Phi\hat{D}_{3j}R^{-1}_{3j}S^\top_{3j}+\Phi\mathbf{D}_1\Phi\bigg)\bigg]\hat{x}dW,\\
	\hat{x}(0)&=x_0.
\end{aligned}\right.
\end{equation}
\end{mypro}

\begin{proof}
Let $\mathbf{Y}=\Phi\hat{x}$ with $\Phi=\begin{pmatrix}
		\Phi_1\\
		\Phi_2\\
		\Phi_3
	\end{pmatrix}$, we have
\[\mathbf{Z}=(I-\Phi\mathbf{D}_2)^{-1}\bigg(\Phi\hat{C}-\sum_{j=1}^3\Phi\hat{D}_{3j}R^{-1}_{3j}S^\top_{3j}+\Phi\mathbf{D}_1\Phi\bigg)\hat{x}=\Psi\hat{x},\]
with $\Psi=\begin{pmatrix}
		\Psi_1\\
		\Psi_2\\
		\Psi_3
	\end{pmatrix}$, and $\Phi$ satisfies (\ref{Phi}). Then the Nash equilibrium (\ref{Nash equilibrium3}) has the above representation (\ref{bar u3j}).
\end{proof}

For (\ref{Gamma 2ij}) to become the incentive strategy of Managerial Level 2, (\ref{bar u3j}) must be matched with the corresponding team-optimal strategies of Executive Level 3 from some conditions. We give the following assumption:
\begin{equation}\label{equal 2}
	\bar{u}_{3j}(t)=u^*_{3j}(t),\qquad t\in[0,T],
\end{equation}
i.e.,
\begin{equation}\label{equal 2'}
	\bar{u}_{3ji}(t)=u^*_{3ji}(t),\qquad t\in[0,T],\qquad j=1,2,3,\;i=1,2.
\end{equation}

\begin{Remark}\label{remark-3.7}
When the relation (\ref{equal 2}) holds, we will get that $\hat{x}(t)=\bar{x}(t)$. In fact, substituting (\ref{u*_2ij}) into (\ref{u*1i equality}), we have
\begin{equation}\label{u*1i}
	u^*_{1i}=\sum_{j=1}^{3}\big[(\xi^*_{1ij}+\eta^*_{1ij}\theta_{2ij})\bar{x}+(\zeta^*_{1ij}+\eta^*_{1ij}\rho_{2ij})u^*_{3ji}\big].
\end{equation}
By (\ref{u*1i}), (\ref{u*_2ij}) and (\ref{bar x'}), we achieve
\begin{equation}
\left\{\begin{aligned}
	d\bar{x}&=\bigg[\bigg(A+\gamma^{-2}EE^\top P+\sum_{i=1}^2\sum_{j=1}^3B_{1i}(\xi^*_{1ij}+\eta^*_{1ij}\theta_{2ij})+\sum_{i=1}^2\sum_{j=1}^3B_{2ij}\theta_{2ij}\bigg)\bar{x}\\
	&\qquad +\sum_{j=1}^3\sum_{i=1}^2\bigg(B_{3ji}+B_{2ij}\rho_{2ij}+B_{1i}(\eta^*_{1ij}\rho_{2ij}+\zeta^*_{1ij})\bigg)u^*_{3ji}\bigg]dt\\
	&\quad +\bigg[\bigg(C+\sum_{i=1}^2\sum_{j=1}^3D_{1i}(\xi^*_{1ij}+\eta^*_{1ij}\theta_{2ij})+\sum_{i=1}^2\sum_{j=1}^3D_{2ij}\theta_{2ij}\bigg)\bar{x}\\
	&\qquad\quad +\sum_{j=1}^3\sum_{i=1}^2\bigg(D_{3ji}+D_{2ij}\rho_{2ij}+D_{1i}(\eta^*_{1ij}\rho_{2ij}+\zeta^*_{1ij})\bigg)u^*_{3ji}\bigg]dW,\\
	\bar{x}(0)&=x_0.
\end{aligned}\right.
\end{equation}
Moreover, $\hat{x}(\cdot)$ satisfies
\begin{equation}
\left\{\begin{aligned}
	d\hat{x}&=\bigg[\bigg(A+\gamma^{-2}EE^\top P+\sum_{i=1}^2\sum_{j=1}^3B_{1i}(\xi^*_{1ij}+\eta^*_{1ij}\theta_{2ij})+\sum_{i=1}^2\sum_{j=1}^3B_{2ij}\theta_{2ij}\bigg)\hat{x}\\
	&\qquad +\sum_{j=1}^3\sum_{i=1}^2\bigg(B_{1i}(\eta^*_{1ij}\rho_{2ij}+\zeta^*_{1ij})+B_{2ij}\rho_{2ij}+B_{3ji}\bigg)\bar{u}_{3ji}\bigg]dt\\
	&\quad +\bigg[\bigg(C+\sum_{i=1}^2\sum_{j=1}^3D_{1i}(\xi^*_{1ij}+\eta^*_{1ij}\theta_{2ij})+\sum_{i=1}^2\sum_{j=1}^3D_{2ij}\theta_{2ij}\bigg)\hat{x}\\
	&\qquad\quad +\sum_{j=1}^3\sum_{i=1}^2\bigg(D_{1i}(\eta^*_{1ij}\rho_{2ij}+\zeta^*_{1ij})+D_{2ij}\rho_{2ij}+D_{3ji}\bigg)\bar{u}_{3ji}\bigg]dW,\\
	\hat{x}(0)&=x_0.
\end{aligned}\right.
\end{equation}
Subtracting the above two equations under the constraint (\ref{equal 2'}), we get
\begin{equation}
\left\{\begin{aligned}
	d(\bar{x}-\hat{x})&=\bigg[A+\gamma^{-2}EE^\top P+\sum_{i=1}^2\sum_{j=1}^3B_{1i}(\xi^*_{1ij}+\eta^*_{1ij}\theta_{2ij})+\sum_{i=1}^2\sum_{j=1}^3B_{2ij}\theta_{2ij}\bigg](\bar{x}-\hat{x})dt\\
	&\quad +\bigg[C+\sum_{i=1}^2\sum_{j=1}^3D_{1i}(\xi^*_{1ij}+\eta^*_{1ij}\theta_{2ij})+\sum_{i=1}^2\sum_{j=1}^3D_{2ij}\theta_{2ij})\bigg](\bar{x}-\hat{x})dW,\\
	(\bar{x}-\hat{x})(0)&=0.
\end{aligned}\right.
\end{equation}
Then, $\bar{x}-\hat{x}\equiv0$ is the solution of this equation, we have
\[\hat{x}(t)=\bar{x}(t),\qquad t\in[0,T],\qquad\mathbb{P}\mbox{-}a.s..\]
When the relation (\ref{equal 2'}) holds, it can be expressed briefly as follows,
\begin{equation}\label{equality 2}
	\bigg[-(\bar{R}_{3ji}+\rho^\top_{2ij}R_{3ij}\rho_{2ij})^{-1}(\rho^\top_{2ij}R_{3ij}\theta_{2ij}+\hat{B}^\top_{3ji}\Phi_j+\hat{D}^\top_{3ji}\Psi_j)+\bar{R}^{-1}_{1ji}(B^\top_{3ji}P+D^\top_{3ji}\Lambda)\bigg]\bar{x}(t)=0.
\end{equation}
\end{Remark}

From the above analysis, we obtain the following theorem.
\begin{mythm}\label{thm-3.2}
Let (A1)-(A5) hold, and suppose the attenuation $\gamma$ is sufficiently large. If there exist matrix-valued functions $\rho^*_{2ij}$ and $\theta^*_{2ij}$ such that (\ref{Phi}) admits solution $\Phi^*(\cdot)$, satisfying equation (\ref{equality 2}), then there exists the incentive strategy set of Managerial Level 2 in the three-level incentive Stackelberg game under the $H_\infty$ constraint with multiple leaders and multiple followers (\ref{state})-(\ref{cost_3j}), given by
\begin{equation}\label{Gamma*2ij}
\begin{aligned}
	\Gamma^*_{2ij}(x(t),u_{3ji}(t),t)=\theta^*_{2ij}(t)x(t)+\rho^*_{2ij}(t)u_{3ji}(t),\\
\end{aligned}
\end{equation}
where
\begin{equation*}
\begin{aligned}
	\theta^*_{2ij}&:=-R^{-1}_{1ij}(B^\top_{2ij}P+D^\top_{2ij}\Lambda)+\rho^*_{2ij}\bar{R}^{-1}_{1ji}(B^\top_{3ji}P+D^\top_{3ji}\Lambda),\\
	\Psi^*&:=(I-\Phi^*\mathbf{D}_2)^{-1}\bigg(\Phi^*\hat{C}-\sum_{j=1}^{3}\Phi^*\hat{D}_{3j}R^{-1}_{3j}S^\top_{3j}+\Phi^*\mathbf{D}_1\Phi^*\bigg).
\end{aligned}
\end{equation*}
\end{mythm}

\begin{proof}
The proof is obvious from what is stated prior to Theorem \ref{thm-3.2}, we omit the details.
\end{proof}

\begin{Remark}
Notice that the following equation holds:
\begin{equation}\label{u*2ij equality}
	u^*_{2ij}(t)=\Gamma^*_{2ij}(\bar{x}(t),u^*_{3ji}(t),t)=\theta^*_{2ij}(t)\bar{x}(t)+\rho^*_{2ij}(t)u^*_{3ji}(t).
\end{equation}
$\rho^*_{2ij}$ and $\theta^*_{2ij}$ depend on the initial state value $x_0$, because the equation (\ref{equality 2}) includes the team-optimal trajectory $\bar{x}(\cdot)$, which depends on $x_0$.
\end{Remark}

\section{A numerical example}

In this section, we give a numerical example with three-level-leader-follower, to demonstrate the effectiveness of our proposed incentive Stackelberg strategy set. Firstly, from the previous analysis in section 3, we give an algorithm procedure to illustrate the process of solving incentive Stackelberg equilibrium strategies.

\begin{algorithm}[H]
\label{Algorithm 1}
\caption{Algorithm procedure of the incentive Stackelberg equilibrium strategies.}
\LinesNumbered
\KwIn{Choose coefficients of the stochastic system (\ref{state})-(\ref{cost_3j}).}
Solve $P$ from equation (\ref{P}), and calculate $\Lambda$.
	
Obtain the optimal state trajectory $\bar{x}$ by (\ref{bar x}).
	
Obtain team-optimal strategy $u^*_c$ and the worst-case disturbance $v^*$ by (\ref{cor}).
\end{algorithm}
\begin{algorithm*}
\LinesNumbered
\KwIn{Equal division $N$, terminal values of $\eta_{1ij}$ and $\zeta_{1ij}$: $\eta_{1ij}(N+1)$, $\zeta_{1ij}(N+1)$.}
\For{k=1 to N}{
	Calculate $\xi_{1ij}(N-k+2)$ via (\ref{Gamma 1i}).
		
	Calculate the values of all matrices in (\ref{notation-1}) and (\ref{notation-2}) at time $N-k+2$.
		
	Solve $\Pi(N-k+1)$ from equation (\ref{Pi equation}), and calculate $\Sigma(N-k+2)$.
		
	Obtain $\eta_{1ij}(N-k+1)$ and $\zeta_{1ij}(N-k+1)$ by (\ref{equality 1}), and calculate $\xi_{1ij}(N-k+1)$.
	}
\KwOut{Incentive strategy parameters $\eta^*_{1ij}$, $\zeta^*_{1ij}$, and $\xi^*_{1ij}$ for Decision-making Level 1.}
\KwIn{Terminal values of $\rho_{2ij}$: $\rho_{2ij}(N+1)$.}
\For{k=1 to N}{
	Calculate $\theta_{2ij}(N-k+2)$ via (\ref{Gamma 2ij}).
		
	Calculate the values of all matrices mentioned in subsection 3.3 at time $N-k+2$.
		
	Solve $\Phi(N-k+1)$ from equation (\ref{Phi}), and calculate $\Psi(N-k+2)$.
		
	Obtain $\rho_{2ij}(N-k+1)$, and calculate $\theta_{2ij}(N-k+1)$.
}
\KwOut{Incentive strategy parameters $\rho^*_{2ij}$, $\theta^*_{2ij}$ for Managerial Level 2.}
	Obtain the incentive strategy sets of Decision-making Level 1 and Managerial Level 2 via (\ref{Gamma*1i}) and (\ref{Gamma*2ij}), respectively.
\end{algorithm*}

\begin{Example}

The system coefficients are given as follows:
$A=0.8$, $C=1$, $E=-1$, $B_{11}=0.65$, $B_{12}=-1$, $D_{11}=0.5$, $D_{12}=1$, $B_{211}=-1$, $B_{212}=0.2$, $B_{213}=0.5$, $B_{221}=0.2$, $B_{222}=-1$, $B_{223}=0.2$, $B_{311}=0.5$, $B_{312}=1$, $B_{321}=0.5$, $B_{322}=0.2$, $B_{331}=0.5$, $B_{332}=0.1$, $D_{211}=-0.1$, $D_{212}=0.1$, $D_{213}=0.2$, $D_{221}=0.5$, $D_{222}=-2$, $D_{223}=0.5$, $D_{311}=0.2$, $D_{312}=1$, $D_{321}=0.2$, $D_{322}=0.1$, $D_{331}=0.5$, $D_{332}=0.1$, $Q_1=1$, $Q_{21}=0.8$, $Q_{22}=0.4$, $Q_{31}=Q_{33}=0.5$, $Q_{32}=0.2$, $R_1=R_2=1$, $R_{111}=R_{112}=0.5$, $R_{113}=1$, $R_{121}=0.5$, $R_{122}=1$, $R_{123}=0.5$, $\bar{R}_{111}=0.5$, $\bar{R}_{112}=0.1$, $\bar{R}_{121}=0.2$, $\bar{R}_{122}=0.1$, $\bar{R}_{131}=0.1$, $\bar{R}_{132}=0.1$, $R_{221}=1$, $R_{212}=0.8$, $R_{213}=1$, $R_{221}=0.9$, $R_{222}=R_{223}=1$, $\bar{R}_{211}=0.3$, $\bar{R}_{212}=0.2$, $\bar{R}_{221}=0.4$, $\bar{R}_{222}=0.1$, $\bar{R}_{231}=\bar{R}_{232}=0.2$, $R_{311}=R_{321}=R_{322}=R_{323}=1$, $R_{312}=0.2$, $R_{313}=0.3$,  $\bar{R}_{311}=\bar{R}_{321}=\bar{R}_{322}=\bar{R}_{332}=1$, $\bar{R}_{312}=2$, $\bar{R}_{331}=0.5$, $G_1=G_{21}=G_{22}=1$, $G_{31}=G_{33}=0.5$, $G_{32}=0.2$, $x_0=0.5$, $T=0.8$, and the disturbance attenuation level $\gamma=1$.

Consider the three-level multi-leader-follower incentive Stackelberg game  with $H_\infty$ constraint by the algorithm. We can give plots of the optimal state trajectory $\bar{x}(\cdot)$ and the team-optimal strategy  $u^*_c=\mathbf{col}\big[u^*_{11}\;u^*_{12}\;u^*_{211}\;u^*_{212}\;u^*_{213}\;u^*_{221}\;u^*_{222}\;u^*_{223}\;u^*_{311}\;u^*_{312}\;u^*_{321}\;u^*_{322}\;u^*_{331}\;u^*_{332}\;\big]$ in Figures 2 and 3, respectively.
In the meanwhile, the worst-case disturbance $v^*(\cdot)$ is plotted in Figure 4.

\begin{figure}[H]
	\centering
	\includegraphics[width=0.75\linewidth]{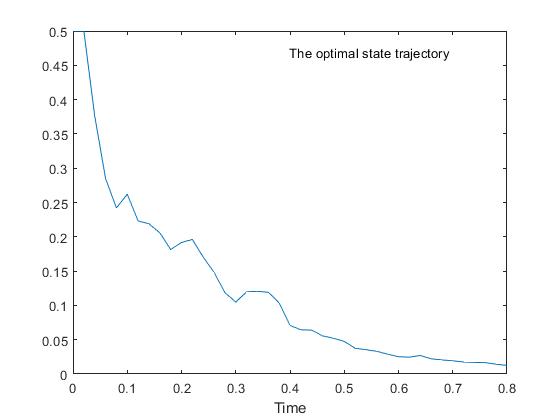}
	\caption{The optimal state trajectory $\bar{x}(\cdot)$.}
\end{figure}

\begin{figure}[H]
	\centering
	\includegraphics[width=0.75\linewidth]{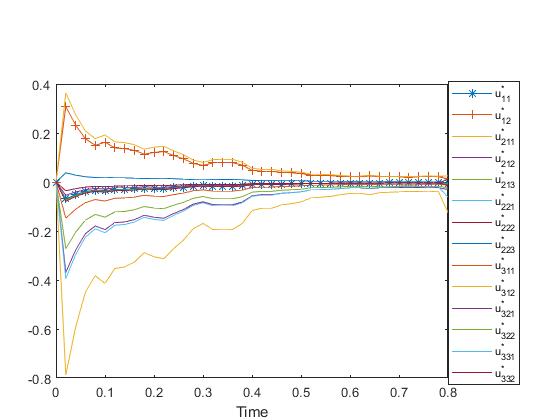}
	\caption{The team-optimal strategy $u^*_c(\cdot)$.}
\end{figure}

\begin{figure}[H]
	\centering
	\includegraphics[width=0.75\linewidth]{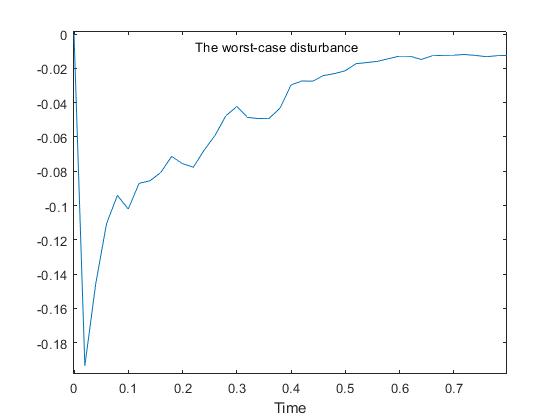}
	\caption{The worst-case disturbance $v^*(\cdot)$.}
\end{figure}	
	
 Avoiding complex calculations, we discuss only the incentive for Decision-making Level 1, Managerial Level 2 being the same as it. Set $N=40$, terminal values $\eta_{1ij}(N+1)$ and $\zeta_{1ij}(N+1)$ are given as follows:
$\eta_{111}(N+1)=5$, $\eta_{112}(N+1)=3$, $\eta_{113}(N+1)=-2$, $\eta_{121}(N+1)=-1$, $\eta_{122}(N+1)=4$, $\eta_{123}(N+1)=1$, $\zeta_{111}(N+1)=1$, $\zeta_{112}(N+1)=-1$, $\zeta_{113}(N+1)=1$, $\zeta_{121}(N+1)=1$, $\zeta_{122}(N+1)=2$, $\zeta_{123}(N+1)=-1$.
Through the algorithm above, we can derive plots of the moduli of incentive strategy parameters $\eta^*_{1ij}$, $\zeta^*_{1ij}$ and $\xi^*_{1ij}$ in Figures 5-7 under the disturbance attenuation level $\gamma=1$. (Since these parameters are complex numbers, for convenience, we only provide the images of their moduli.)

The incentive intensities are very flat in the early stages, but they increase significantly in the end to achieve the team-optimal strategy.

\begin{figure}[H]
	\centering
	\includegraphics[width=0.75\linewidth]{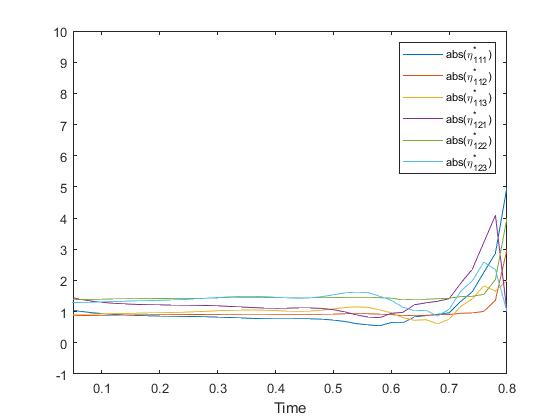}
	\caption{The moduli of incentive strategy parameters $\eta^*_{1ij}(\cdot)$.}
\end{figure}

\begin{figure}[H]
	\centering
	\includegraphics[width=0.75\linewidth]{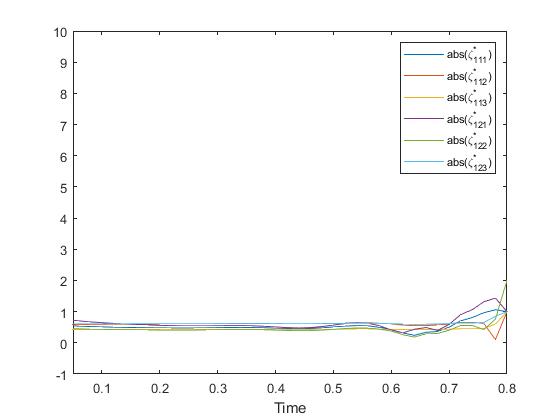}
	\caption{The moduli of incentive strategy parameters $\zeta^*_{1ij}(\cdot)$.}
\end{figure}

\begin{figure}[H]
	\centering
	\includegraphics[width=0.75\linewidth]{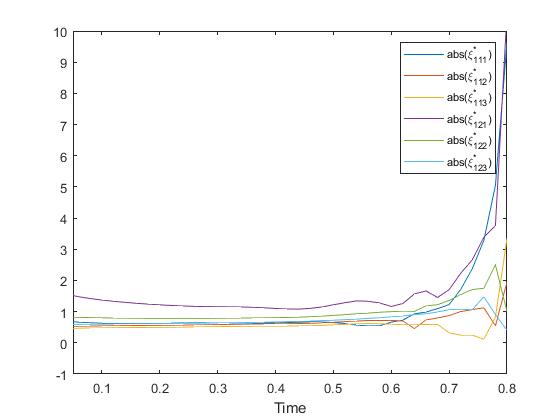}
	\caption{The moduli of incentive strategy parameters $\xi^*_{1ij}(\cdot)$.}
\end{figure}

We can investigate the impacts of the disturbance attenuation level $\gamma$ on the optimal state trajectory $\bar{x}(\cdot)$ and team-optimal strategy $u^*_c(\cdot)$. When the disturbance attenuation level $\gamma$ takes these values in $\{1,10,100\}$, the change of the worst-case disturbance $v^*(\cdot)$ is plotted in Figure 8, and Figures 9 and 10 depict the changes of the optimal state trajectory $\bar{x}(\cdot)$ and team-optimal strategy $u^*_c(\cdot)$.

In Figure 8, the higher disturbance attenuation level means that the measure of the worst possible impact of stochastic disturbance on the system that players can accept is larger. In other words, players are not conservative and remain optimistic about unknown risks. As a result, the worst-case disturbance $v^*(\cdot)$ they consider will be smaller. So, the intensity of the team-optimal strategy $u^*_c(\cdot)$ they give will also be relatively small (Figure 10).

\begin{figure}[H]
	\centering
	\includegraphics[width=0.75\linewidth]{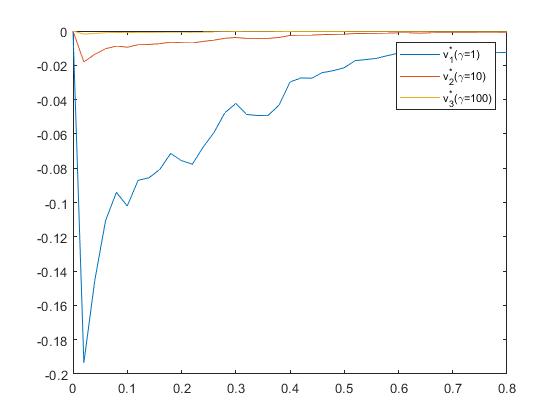}
	\caption{The impact of $\gamma$ on the worst-case disturbance $v^*(\cdot)$.}
\end{figure}

\begin{figure}[H]
	\centering
	\includegraphics[width=0.75\linewidth]{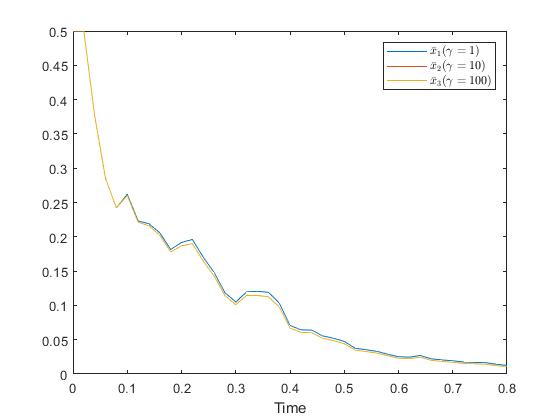}
	\caption{The impact of $\gamma$ on the optimal state trajectory $\bar{x}(\cdot)$.}
\end{figure}

\begin{figure}[H]
	\centering
	\includegraphics[width=0.75\linewidth]{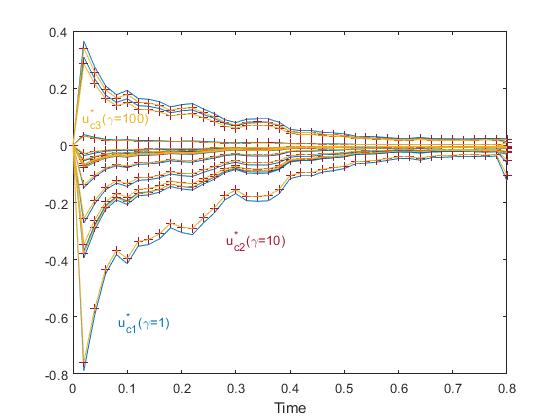}
	\caption{The impact of $\gamma$ on the team-optimal strategy $u^*_c(\cdot)$.}
\end{figure}

Finally, we consider the impacts of terminal values of $\eta_{1ij}$ and $\zeta_{1ij}$ on the moduli of incentive strategy parameters $\eta^*_{1ij}$, $\zeta^*_{1ij}$ and $\xi^*_{1ij}$. Consider the intensities (absolute values) of terminal values $\eta_{1ij}(N+1)$ and $\zeta_{1ij}(N+1)$ become 85\% of the original under the same disturbance attenuation level $\gamma=1$, which are given as follows:
$\eta_{111}(N+1)=4.25$, $\eta_{112}(N+1)=2.55$, $\eta_{113}(N+1)=-1.7$, $\eta_{121}(N+1)=-0.85$, $\eta_{122}(N+1)=3.4$, $\eta_{123}(N+1)=0.85$, $\zeta_{111}(N+1)=0.85$, $\zeta_{112}(N+1)=-0.85$, $\zeta_{113}(N+1)=0.85$, $\zeta_{121}(N+1)=0.85$, $\zeta_{122}(N+1)=1.7$, $\zeta_{123}(N+1)=-0.85$.
By the algorithm procedure, we get plots of the corresponding moduli of incentive strategy parameters $\eta^*_{1ij}$, $\zeta^*_{1ij}$ and $\xi^*_{1ij}$ in Figure 11-13.

\begin{figure}[H]
	\centering
	\includegraphics[width=0.75\linewidth]{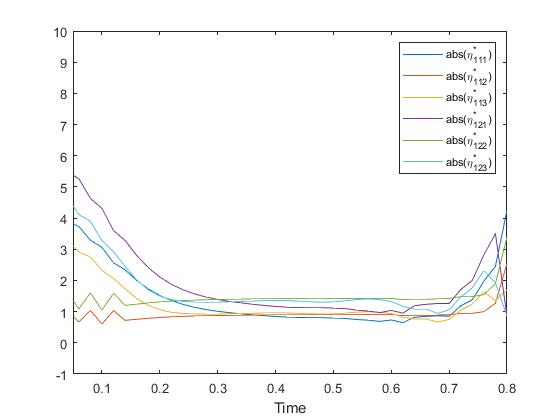}
	\caption{The impacts of $\eta_{1ij}(N+1)$ and $\zeta_{1ij}(N+1)$ on the moduli of $\eta^*_{1ij}(\cdot)$.}
\end{figure}

\begin{figure}[H]
	\centering
	\includegraphics[width=0.75\linewidth]{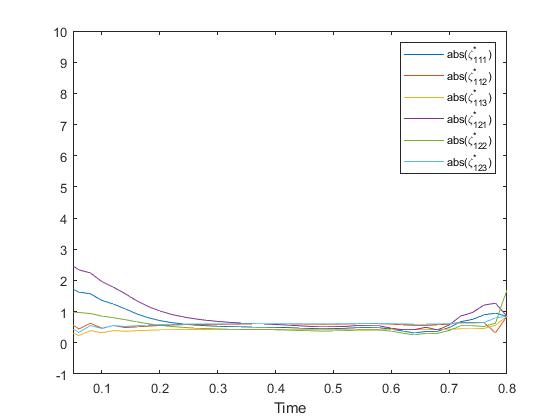}
	\caption{The impacts of $\eta_{1ij}(N+1)$ and $\zeta_{1ij}(N+1)$ on the moduli of $\zeta^*_{1ij}(\cdot)$.}
\end{figure}

\begin{figure}[H]
	\centering
	\includegraphics[width=0.75\linewidth]{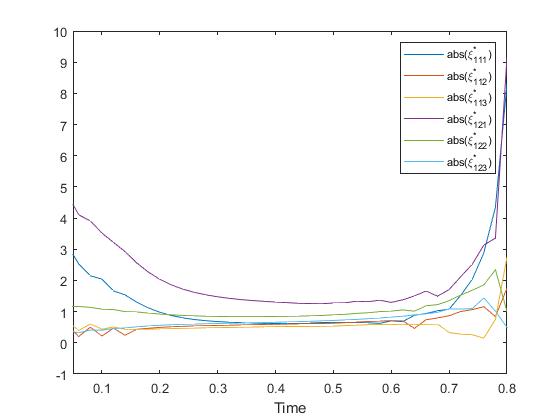}
	\caption{The impacts of $\xi_{1ij}(N+1)$ and $\xi_{1ij}(N+1)$ on the moduli of $\xi^*_{1ij}(\cdot)$.}
\end{figure}
 Compared with Figure 5-7, to make the intensities of terminal incentive parameters smaller, the incentive strategy parameters' intensities have to be larger throughout the process in order to achieve the same goal under the same system.
\end{Example}

\begin{Remark}
	We have to emphasise that the algorithm procedure of solving incentive Stackelberg equilibrium strategies depends on the terminal values of incentive strategy parameters $\eta_{1ij}$, $\zeta_{1ij}$ and $\rho_{2ij}$ to be determined, which is a drawback of the algorithm.
\end{Remark}

\section{Concluding remarks}

In this paper, we have studied a three-level multi-leader-follower incentive Stackelberg game with $H_\infty$ constraint, where the control variables and the external disturbance enter the diffusion term and drift term of the state equation, respectively. Via $H_2/H_\infty$ control theory, convex analysis theory,  maximum principle and decoupling technique, the three-level incentive Stackelberg strategy set is given. We derive sufficient conditions for the three-level incentive Stackelberg game, and show that three-level incentive matrices depend on an initial state value $x_0$ and the corresponding trajectory is equivalent to the optimal-team trajectory $\bar{x}(\cdot)$ after achieving incentive. Numerical simulations are also given.

In the future, we will consider multi-leader-follower incentive Stackelberg differential game with partial observation, where both the leader and the followers have their own observation equations. Linear-quadratic mean-field type multi-leader-follower incentive Stackelberg differential game is also an interesting and challenging topic.


\begin{thebibliography}{0}

\bibitem{Ahmed-Mukaidani16} M. Ahmed, H. Mukaidani, $H_\infty$-constrained incentive Stackelberg game for discrete time systems with multiple non-cooperative followers, \emph{IFAC-PapersOnLine}, 49, 262-267, 2016.

\bibitem{Ahmed-Mukaidani-Shima17} M. Ahmed, H. Mukaidani, and T. Shima, $H_\infty$-constrained incentive Stackelberg games for discrete-time stochastic systems with multiple followers, \emph{IET Control Theory Appl.}, 11, 2475-2485, 2017.
		
\bibitem{Ahmed-Mukaidani-Shima17'} M. Ahmed, H. Mukaidani, and T. Shima, Infinite-horizon multi-leader-follower incentive stackelberg games for linear stochastic systems with $H_\infty$ constraint, \emph{Proc. SICE Annual Conference 2017}, 1202-1207, September 19-22, Kanazawa, Japan, 2017.

\bibitem{Bagchi-Basar81} A. Bagchi, T. Ba\c{s}ar, Stackelberg strategies in linear-quadratic stochastic differential games, \emph{J. Optim. Theory Appl.}, 35, 443-464, 1981.
		
\bibitem{Basar-Olsder99} T. Ba\c{s}ar, G.J. Olsder, \emph{Dynamic Noncooperative Game Theory (Classics in Applied Mathematics)}, SIAM, Philadelphia, USA, 1999.
		
\bibitem{Bensoussan-Chen-Sethi15} A. Bensoussan, S. Chen, and S.P. Sethi. The maximum principle for global solutions of stochastic Stackelberg differential games, \emph{SIAM J. Control Optim.}, 53, 1956-1981, 2015.

\bibitem{Chen-Zhang04} B.-S. Chen, W. Zhang, Stochastic $H_2/H_\infty$ control with state-dependent noise, \emph{IEEE Trans. Autom. Control.}, 49, 45-57, 2004.
		
\bibitem{Feng-Hu-Huang24} X. Feng, Y. Hu, and J. Huang, Linear-quadratic two-person differential game: Nash game versus Stackelberg game, local information versus global information, \emph{ESAIM: Control Optim. Calc. Var.}, 30, 47, 2024.

\bibitem{Gao-Lin-Zhang23} W. Gao, Y. Lin, and W. Zhang, Incentive feedback Stackelberg strategy for the discrete-time stochastic systems, \emph{J. Frankl Inst.}, 360, 2404-2420, 2023.
		
\bibitem{Gao-Lin-Zhang24} W. Gao, Y. Lin, and W. Zhang, Incentive feedback Stackelberg strategy in mean-field type stochastic difference games, \emph{J. Syst. Sci. Complex.}, 37, 1425-1445, 2024.

\bibitem{Ho-Luh-Olsder82} Y. Ho, P. Luh, and G.J. Olsder, A control-theoretic view of incentives, \emph{Automatica}, 18, 167-180, 1982.

\bibitem{Ishida87} T. Ishida, Three-level incentive schemes using follower's strategies in differential games, \emph{Int. J. Control}, 46, 1739-1750, 1987.

\bibitem{Ishida-Shimemura83} T. Ishida, E. Shimemura, Three-level incentive strategies in differential games, \emph{Int. J. Control}, 38, 1135-1148, 1983.

\bibitem{Kang-Shi22} K. Kang, J. Shi, A three-level stochastic linear-quadratic Stackelberg differential game with asymmetric information, \emph{Proc. 35th Chinese Control and Decision Conference}, 5629-5636, May 20-22, Yichang, China, 2023.

\bibitem{Kawakami-Mukaidani-Xu-Tanaka18} K. Kawakami, H. Mukaidani, H. Xu, and Y. Tanaka, Incentive Stackelberg-Nash strategy with disturbance attenuation for stochastic LPV systems, \emph{Proc. 2018 IEEE International Conference on Systems, Man, and Cybernetics}, 3951-3955, October 7-10, Miyazaki, Japan, 2018.

\bibitem{Li-Cruz-Simaan02} M. Li, J.B. Cruz Jr., and M.A. Simaan, An approach to discrete-time incentive feedback Stackelberg games, \emph{IEEE Trans. Systems, Man, and Cybernetics}, 32, 10-24, 2002.

\bibitem{Li et al21} Z. Li, D. Marelli, M. Fu, Q. Cai, and W. Meng, Linear quadratic Gaussian Stackelberg game under asymmetric information patterns, \emph{Automatica}, 125, 109406, 2021.
	
\bibitem{Li-Shi24} Z. Li, J. Shi, Closed-loop solvability of linear quadratic mean-field type Stackelberg stochastic differential games, \emph{Appl. Math. Optim}, 90, 22, 2024.
	
\bibitem{Li-Yu18} N. Li, Z. Yu, Forward-backward stochastic differential equations and linear-quadratic generalized Stackelberg games, \emph{SIAM J. Control Optim.}, 56, 4148-4180, 2018.

\bibitem{Limebeer94} D.J.N. Limebeer, B.D.O. Anderson, and B. Hendel, A Nash game approach to mixed $H_2/H_\infty$ control, \emph{IEEE Trans. Autom. Control.}, 39, 69-82, 1994.

\bibitem{Lin-Gao-Zhang22} Y. Lin, W. Gao, and W. Zhang, Incentive feedback Stackelberg strategy for stochastic system with state-dependent noise, \emph{J. Frankl. Inst.}, 359, 2058-2072, 2022.
		
\bibitem{Mizukami-Wu87} K. Mizukami, H. Wu, Two-level incentive Stackelberg strategies in LQ differential games with two noncooperative leaders and one follower, \emph{Trans. Society Instrument Control Engineers}, 23, 625-632, 1987.
		
\bibitem{Mizukami-Wu88} K. Mizukami, H. Wu, Incentive Stackelberg strategies in linear quadratic differential games with two noncooperative followers, System Modelling and Optimization, 113, \emph{Lecture Notes in Control and Information Sciences}, 436-445, Berlin Heidelberg, Springer, 1988.

\bibitem{Moon-Basar18} J. Moon, T. Ba\c{s}ar, Linear quadratic mean field Stackelberg differential games, \emph{Automatica}, 97, 200-213, 2018.

\bibitem{Mukaidani-Irie-Xu-Zhang22} H. Mukaidani, S. Irie, H. Xu, and W. Zhang, Robust incentive Stackelberg games with a large population for stochastic mean-field systems, \emph{IEEE Control Systems Letters}, 6, 1934-1939, 2022.

\bibitem{Mukaidani-Saravanakumar-Xu20} H. Mukaidani, R. Saravanakumar, and H. Xu, Robust incentive Stackelberg strategy for Markov jump linear stochastic systems via static output feedback, \emph{IET Control Theory Appl.}, 14, 1246-1254, 2020.

\bibitem{Mukaidani-Shima-Unno-Xu-Dragan17} H. Mukaidani, T. Shima, M. Unno, H. Xu, and V. Dragan, Team-optimal incentive Stackelberg strategies for Markov jump linear stochastic systems with $H_\infty$ constraint, \emph{IFAC-PapersOnLine}, 50, 3780-3785, 2017.
		
\bibitem{Mukaidani-Xu18} H. Mukaidani, H. Xu, Robust incentive Stackelberg games for stochastic LPV systems, \emph{Proc. 2018 IEEE Conference on Decision and Control}, 1059-1064, December 17-19, Miami Beach, USA, 2018.
		
\bibitem{Mukaidani-Xu19} H. Mukaidani, H. Xu, Incentive Stackelberg games for stochastic linear systems with $H_\infty$ constraint, \emph{IEEE Trans. on Cybern}, 49, 1463-1474, 2019.

\bibitem{Mukaidani-Xu-Dragan18} H. Mukaidani, H. Xu, and V. Dragan, Static output-feedback incentive Stackelberg game for discrete-time Markov jump linear stochastic systems with external disturbance, \emph{IEEE Control Systems Letters}, 2, 701-706, 2018.

\bibitem{Mukaidani-Xu-Shima-Ahmed18} H. Mukaidani, H. Xu, T. Shima, and M. Ahmed, Multi-leader-follower incentive Stackelberg game for infinite-horizon Markov jump linear stochastic systems with $H_\infty$ constraint, \emph{Proc. 2018 IEEE International Conference on Systems, Man, and Cybernetics}, 3956-3963, October 7-10, Miyazaki, Japan, 2018.

\bibitem{Mukaidani-Xu-Shima-Dragan17} H. Mukaidani, H. Xu, T. Shima, and V. Dragan, A stochastic multiple-leader-follower incentive Stackelberg strategy for Markov jump linear systems, \emph{IEEE Control Systems Letters}, 1, 250-255, 2017.

\bibitem{Shi-Wang-Xiong16} J. Shi, G. Wang, and J. Xiong, Leader-follower stochastic differential game with asymmetric information and applications, \emph{Automatica}, 63, 60-73, 2016.
		
\bibitem{Shi-Wang-Xiong17} J. Shi, G. Wang, and J. Xiong, Linear-quadratic stochastic Stackelberg differential game with asymmetric information, \emph{Sci. China Infor. Sci.}, 60, 092202, 2017.

\bibitem{von Stackelberg52} H. von Stackelberg, \emph{Marktform und Gleichgewicht}, Springer, Vienna, 1934. (An English translation appeared in \emph{The Theory of The Market Economy}, Oxford University Press, UK,
		1952.)

\bibitem{Sun-Yong20} J. Sun, J. Yong, Stochastic linear-quadratic optimal control theory: Differential games and mean-field problems, \emph{Springer Briefs in Mathematics}, 2020.

\bibitem{Wang24} B. Wang, Leader-follower mean field LQ games: A direct approach, \emph{Asian J. Control}, 26, 617-625, 2024.
		
\bibitem{Wang-Wang24} Y. Wang, W. Wang, Partially observed mean-field Stackelberg stochastic differential game with two followers, \emph{Inter. J. Control}, 97, 1999-2008, 2024.

\bibitem{Yong02} J. Yong, A leader-follower stochastic linear quadratic differential game, \emph{SIAM J. Control Optim.}, 41, 1015-1041, 2002.
		
\bibitem{Zheng-Basar82} Y. Zheng, T. Ba\c{s}ar, Existence and derivations of optimal affine incentive schemes for Stackelberg games with partial information: A geometric approach, \emph{Int. J. Control}, 35, 997-1011, 1982.
		
\bibitem{Zheng-Basar-Cruz84} Y. Zheng, T. Ba\c{s}ar and J. B. Cruz Jr., Stackelberg strategies and incentives in multiperson deterministic decision problems, \emph{IEEE Trans. Systems, Man, and Cybernetics}, 14, 10-24, 1984.
	
\bibitem{Zheng-Shi22} Y. Zheng, J. Shi, Stackelberg stochastic differential game with asymmetric noisy observations, \emph{Inter. J. Control}, 95, 2510-2530, 2022.

\end{thebibliography}
\end{document}